\documentclass[onefignum,onetabnum]{siamart220329}
\usepackage{graphics,graphicx,epstopdf,url}
\usepackage{arydshln}
\usepackage{color}
\usepackage{geometry}
\usepackage{algorithmic}
\usepackage{url,cases,supertabular}
\usepackage{amssymb,mathtools}
\usepackage{enumerate,makecell}
\usepackage{amsfonts}
\usepackage{graphicx,epstopdf}
\usepackage{color,verbatim}
\usepackage{mathrsfs,subeqnarray,subfigure}
\usepackage{listings,array}
\usepackage{booktabs}
\usepackage{dashrule}
\usepackage{arydshln}
\usepackage{graphicx}
\usepackage{amsfonts}
\usepackage{mathrsfs}
\usepackage{graphicx}  \usepackage{epstopdf}
\usepackage{subfigure}  \usepackage{lineno}
\usepackage{multirow,bm}
\usepackage{color}
\usepackage{algorithm}
\usepackage{algorithmic}
\usepackage{longtable}
\usepackage{mdwtab}
\usepackage{mdwmath}
\usepackage{bbding}
\usepackage{booktabs}
\usepackage{rotating}
\usepackage{enumitem}
\usepackage{lineno}
\usepackage{url}
\usepackage{pdflscape}
\usepackage{makecell}

\numberwithin{equation}{section}

\newtheorem{assumption}[theorem]{Assumption}

\newcommand{\ba}{\begin{array}}
\newcommand{\ea}{\end{array}}

\newcommand{\bit}{\begin{itemize}}
\newcommand{\eit}{\end{itemize}}
\newcommand{\be}{\begin{equation}}
\newcommand{\ee}{\end{equation}}
\newcommand{\bee}{\begin{equation*}}
\newcommand{\eee}{\end{equation*}}
\newcommand{\bea}{\begin{eqnarray}}
\newcommand{\eea}{\end{eqnarray}}

\newcommand{\st}{\mathrm{s.t.}}

\newcommand{\bs}{\mathrm{\bf s}}
\newcommand{\bx}{\mathbf{x}}
\newcommand{\bu}{\mathbf{u}}
\newcommand{\bv}{\mathbf{v}}
\newcommand{\bW}{\mathbf{W}}

\newcommand{\Rmn}[1]{\uppercase\expandafter{\romannumeral#1}}
\renewcommand{\theenumi}{\roman{enumi}}

\newcommand{\Rcal}{\mathcal{R}}

\newcommand{\Pcal}{\mathcal{P}}

\numberwithin{equation}{section}

\newcommand{\Mcal}{\mathcal{M}}

\newcommand{\grad}{\mathrm{grad}}

\newcommand{\R}{\mathbb{R}}

\numberwithin{theorem}{section}
\newcommand{\iprod}[2]{\left \langle #1, #2 \right \rangle }

\usepackage{lipsum}
\usepackage{clrscode}
\usepackage{appendix}
\usepackage{bm}
\usepackage{booktabs}
\usepackage{url}
\usepackage{multirow}
\usepackage{textcomp}
\usepackage{amsmath}
\usepackage{amsfonts}
\usepackage{amssymb}
\usepackage{mathrsfs}
\usepackage{hyperref}
\usepackage{graphicx,graphics,subfigure}
\usepackage{hyperref,url}
\usepackage{epsf,epstopdf}
\usepackage{algorithm}
\usepackage{algorithmic}
\usepackage{bbm}
\usepackage{dsfont}
\usepackage{indentfirst}
\usepackage{pdfpages}
\usepackage{pdflscape}
\usepackage{longtable}
\ifpdf
  \DeclareGraphicsExtensions{.eps,.pdf,.png,.jpg}
\else
  \DeclareGraphicsExtensions{.eps}
\fi
\allowdisplaybreaks[1]

\headers{Decentralized projected Riemannian gradient method}{Kangkang Deng, Jiang Hu}

\title{Decentralized projected Riemannian gradient method for smooth optimization on compact submanifolds}

\author{Kangkang Deng\thanks{ Department of Mathematics,  National University of Defense Technology, Changsha, 410073,
China (\email{freedeng1208@gmail.com}).}
\and Jiang Hu\thanks{Corresponding author. Massachusetts General Hospital and Harvard Medical School, Harvard University, Boston, MA 02114
(\email{hujiangopt@gmail.com}).}
}

\usepackage{amsopn}

\ifpdf
\hypersetup{
  pdftitle={Decentralized projected Riemannian gradient method},
  pdfauthor={Kangkang Deng, Jiang Hu}
}
\fi

\allowdisplaybreaks[2]
\begin{document}

\maketitle

\begin{abstract}
We consider the problem of decentralized nonconvex optimization over a compact submanifold, where each local agent's objective function defined by the local dataset is smooth. Leveraging the powerful tool of proximal smoothness, we establish local linear convergence of the projected gradient descent method with a unit step size for solving the consensus problem over the nonconvex compact submanifold. This serves as the basis for designing and analyzing decentralized algorithms on manifolds. Subsequently, we propose two decentralized methods: the decentralized projected Riemannian gradient descent (DPRGD) and the decentralized projected Riemannian gradient tracking (DPRGT).  We establish their convergence rates of $\mathcal{O}(1/\sqrt{K})$ and $\mathcal{O}(1/K)$, respectively, to reach a stationary point. To the best of our knowledge, DPRGT is the first decentralized algorithm to achieve exact convergence for solving decentralized optimization over a compact submanifold. Beyond the linear convergence results on the consensus, two key tools developed in the proof are the Lipschitz-type inequality of the projection operator and the Riemannian quadratic upper bound for smooth functions on the compact submanifold, which could be of independent interest. Finally, we demonstrate the effectiveness of our proposed methods compared to state-of-the-art ones through numerical experiments on eigenvalue problems and low-rank matrix completion.
\end{abstract}


\section{Introduction}

Decentralized optimization has gained significant attention in recent years due to its potential applications in large-scale distributed systems, such as sensor networks, distributed computing systems, and machine learning. In these systems, data is often distributed across multiple agents or nodes, and a centralized optimization approach may not be feasible due to issues such as privacy concerns and limited computational resources. In this paper, we consider the distributed smooth optimization over a compact submanifold
\be \label{prob:original}
\begin{aligned}
  \min \quad & \frac{1}{n}\sum_{i=1}^n f_i(x_i), \\
  \st \quad & x_1 = \cdots = x_n, \;\; x_i \in \Mcal, \;\; \forall i=1,2,\ldots, n,
\end{aligned}
\ee
where $n$ is the number of agents, $f_i$ is the local objective at the $i$-th agent, and $\Mcal$ is a compact smooth submanifold of $\R^{d\times r}$, e.g., the Stiefel manifold ${\rm St}(d,r):=\{ x \in \R^{d\times r} : x^\top x = I_r \}$. Problem \eqref{prob:original} widely exists in machine learning, signal processing, and deep learning, see, e.g., the principal component analysis \cite{ye2021deepca}, the low-rank matrix completion \cite{boumal2015low,kasai2019riemannian}, the low-dimension subspace learning \cite{ando2005framework,mishra2019riemannian}, and the deep neural networks with batch normalization \cite{cho2017riemannian,hu2022riemannian}. 

Decentralized optimization in the Euclidean space (i.e., $\Mcal = \R^{d\times r}$) has been extensively studied over the past few decades. The decentralized (sub)-gradient descent (DGD) method \cite{tsitsiklis1986distributed,nedic2009distributed,yuan2016convergence} provides a straightforward approach to combine local gradient update and consensus update. However, the analysis therein indicates that DGD can not converge to a stationary point (i.e., exact convergence) when fixed step sizes are used. When $f_i$'s are convex, to achieve the exact convergence, the local historic information is investigated in various papers, such as EXTRA \cite{shi2015extra}, DLM \cite{ling2015dlm}, exact diffusion \cite{yuan2018exact}, NIDS \cite{li2019decentralized}, and the gradient tracking-type methods \cite{xu2015augmented,qu2017harnessing}. The exact convergence is shown with the use of fixed step sizes. If the objective functions are nonconvex, there are also extensive studies, see, e.g., \cite{bianchi2012convergence,di2016next,tatarenko2017non,wai2017decentralized,hong2017prox,zeng2018nonconvex,scutari2019distributed,sun2020improving}.  One can also regard problem \eqref{prob:original} as a decentralized composite optimization problem, where the nonsmooth part is the indicator function of the manifold. Then, decentralized proximal gradient-type algorithms can be applied if the projection to the manifold is available. Note that the previous studies, e.g., \cite{bianchi2012convergence,di2016next,zeng2018nonconvex}, require the convexity of the regularizer or at least the convexity of its domain to define the average point and achieve the consensus. 
However, since the manifold $\Mcal$ is nonlinear and nonconvex, these works will fail when applied to solve problem \eqref{prob:original}.

As a special case of problem \eqref{prob:original}, the Riemannian consensus
problem (defined in \eqref{prob:consensus}) is well-studied \cite{tron2012riemannian,sarlette2009consensus,markdahl2020high,chen2021local}. The analysis of decentralized gradient algorithms for solving \eqref{prob:original} often relies on the linear convergence of consensus iteration for solving such consensus problem.
In the Euclidean setting, due to its linear and convex structure, the linear convergence of the consensus can be straightforwardly obtained by applying the (projected) gradient method.
As explained above, the nonconvex nature of the problem in \eqref{prob:original} significantly distinguishes it from Euclidean decentralized composite optimization problems, thereby making consensus construction and algorithmic design more challenging. To address the nonconvexity of the manifold $\Mcal$, the authors \cite{tron2012riemannian} introduce a Riemannian consensus based on the geodesic distance on the manifold. However, to solve the resulting consensus problem, expensive operations such as exponential maps and vector transports are required. Recently, for the case where $\Mcal$ is the Stiefel manifold, a more tractable consensus has been defined in \cite{chen2021local} by using the Euclidean distance. By investigating the restricted secant inequality, they show that the Riemannian gradient method will converge linearly to the consensus points when multiple-step consensuses are used in each iteration. However, for general compact submanifolds, these results are not applicable because they rely on the spectral properties specific to orthogonal matrices.

Building on the consensus results mentioned earlier, several decentralized Riemannian algorithms have been developed \cite{shah2017distributed,mishra2019riemannian}. However, these methods require an asymptotically infinite number of consensus steps for convergence, which limits their practical applicability. For the case where $\Mcal$ is the Stiefel manifold, the authors \cite{chen2021decentralized} propose a decentralized Riemannian gradient descent method and its gradient-tracking version. 
When the local objective function is of negative log-probability type, a decentralized Riemannian natural gradient method is presented in \cite{hu2023decentralized}. To use a single step of consensus, augmented Lagrangian methods \cite{wang2022decentralized,wang2022variance} are also investigated, where a different stationarity is used. However, these studies rely on the orthogonal structure of the Stiefel manifold. 

\subsection{Contribution} 
This paper involves proving the linear convergence of consensus error and, based on this, developing two decentralized gradient-type algorithms to solve \eqref{prob:original}. We summarize our contributions as follows: 
\begin{itemize}
    \item \textbf{Linear consensus of the projected gradient method.} We present an Euclidean distance-based consensus problem as minimizing a quadratic function over smooth submanifolds (defined in \eqref{prob:consensus}). To address the inherent nonconvexity from the submanifold constraint, we leverage the powerful tool of proximal smoothness from variational analysis, establishing the linear convergence of the projected gradient descent within a neighborhood with explicit characterization (see Section \ref{sec:linear-consensus}). Such results are instrumental for analyzing decentralized manifold optimization algorithms. Compared to the results for the Stiefel manifold \cite{chen2021local}, our analysis allows the use of the unit step size and is applicable to a broader range of compact smooth submanifolds. Notably, our approach can be extended to other nonconvex sets or functions, such as strongly prox-regular functions \cite{hu2022projected}.
    
    \item \textbf{Derivation of essential equalities on the compact submanifold.}     
    We derive a Riemannian quadratic upper bound for differentiable functions on compact submanifolds with locally Lipschitz continuous gradients. This serves as a key tool for analyzing the convergence of Riemannian gradient-type algorithms, paralleling its counterpart in the Euclidean space. Using the proximal smoothness, we present a Lipschitz-type inequality for the projection operator on compact manifolds, analogous to the Lipschitz-type inequalities of the retraction operator in \cite{boumal2019global}, playing a central role in analyzing projected gradient methods. 

    \item \textbf{The first decentralized algorithms for solving \eqref{prob:original} with exact convergence.} 
    Building upon the linear consensus of the projected gradient method, we propose a decentralized projected Riemannian gradient descent method (Algorithm \ref{alg:drpgd}) for solving \eqref{prob:original}, where each step combines a projected gradient step for the consensus and a Riemannian gradient step for $f$. We show that the iteration complexity of obtaining an $\epsilon$-stationary point (see Definition \ref{def:station}) is $\mathcal{O}(\epsilon^{-2})$ (see Theorem \ref{thm:dprgd}). To achieve exact convergence with a constant step size, we propose a decentralized projected Riemannian gradient tracking method (Algorithm \ref{alg:drgta}) and establish an improved iteration complexity of $\mathcal{O}(\epsilon^{-1})$ (see Theorem \ref{thm:dprgt}). Numerical results on (generalized) eigenvalue problems and low-rank matrix completion demonstrate the effectiveness of our proposed method compared with state-of-the-art ones. 
\end{itemize}

\subsection{Notation}
Let $\mathbf{1}_n\in \mathbb{R}^n$ be a vector of all entries one. Define $J := \frac{1}{n}\mathbf{1}_n\mathbf{1}_n^{\top}$.   Let $\bx := [x_1^\top, \cdots, x_n^\top]^\top$ denote the collection of all local variables $x_i$. 
 Define $f(\bx): = \frac{1}{n}\sum_{i=1}^n f_i(x_i)$. Let $\bW^t := W^t \otimes I_d \in \R^{nd \times nd}$, where $t$ is a positive integer and $\otimes$ denotes the Kronecker product. For the compact submanifold $\Mcal$ of $\mathbb{R}^{d \times r}$, we always take the Euclidean metric $\iprod{\cdot}{\cdot}$ as the Riemannian metric. We use $\|\cdot \|$ to denote the Euclidean norm. We denote the $n$-fold Cartesian product of $\Mcal$ with itself as $\Mcal^n = \Mcal \times \cdots \times \Mcal$, and use $[n] := \{1, 2, \cdots, n\}$.
For any $x\in \Mcal$, the tangent space and normal space of $\Mcal$ at $x$ are denoted by $T_x\Mcal$ and $N_x\Mcal$, respectively. We define $\nabla f(x)$ as the Euclidean gradient of $f$, and $\grad f(x)$ as the Riemannian gradient of $f$.

\section{Preliminary}
 We define the distance and the nearest-point projection of a point $y \in \mathbb{R}^{d\times r}$ onto $\Mcal$ by
 $$
 \text{dist}(x,\Mcal) : = \inf_{y\in \Mcal} \| y - x\|, ~~ \text{and}~~ \Pcal_{\mathcal{M}}(x) : = \arg\min_{y\in \mathcal{M}} \| y - x\|,
 $$
 respectively. For any real number $R >0$, we define the $R$-tube around $\mathcal{M}$ to be the set:
 $$
 U_{\mathcal{M}}(R): = \{x:~   \text{dist}(x,\mathcal{M}) < R\}.
 $$
We say a closed set $\mathcal{M}$ is $R$-proximally smooth if the projection $\Pcal_{\mathcal{M}}(x)$ is a singleton whenever $\text{dist}(x,\mathcal{M}) < R$. Following \cite{clarke1995proximal}, an $R$-proximally smooth set $\mathcal{M}$ satisfies that: 
\begin{itemize}
    \item[(i)] For any real $\gamma\in (0,R)$, the estimate holds:
    \be \label{eq:lip-proj}
    \left\| \Pcal_{\mathcal{M}} (x) -\Pcal_{\mathcal{M}} (y)\right\| \leq \frac{R}{R-\gamma}\|x - y\|,~~ \forall x,y \in \bar{U}_{\mathcal{M}}(\gamma),
    \ee
    where $\bar{U}_{\Mcal}(\gamma):=\{x: {\rm dist}(x, \Mcal) \leq \gamma \}$.
    \item[(ii)] For any point $x\in \mathcal{M}$ and a normal $v\in N_x\Mcal$, the following inequality holds for all $y\in\Mcal$:
    \be \label{proximally-0}
    \iprod{v}{y-x} \leq \frac{\|v\|}{2R}\|y-x\|^2.
    \ee
\end{itemize}
It is shown that any compact $C^2$-submanifolds in Euclidean space belong to proximally smooth set \cite{clarke1995proximal,balashov2021gradient,davis2020stochastic}. For example, the Stiefel manifold is a $1$-proximally smooth set \cite{balashov2021gradient}.  
As we will see later, those properties of the proximally smooth set will be crucial for the design and analysis of the algorithms. 
 
In the design of manifold optimization algorithms \cite{absil2009optimization,boumal2023introduction,hu2020brief,huang2017intrinsic}, a key concept is the retraction operator. A smooth mapping $\Rcal:T\mathcal{M}:=\cup_{x \in \Mcal} T_{x} \Mcal \rightarrow \Mcal$ is called a retraction operator if  
\begin{itemize}
	\item $\Rcal_{x}(0) = x$,
	\item $\mathrm{D}\Rcal_{x}(0)[\xi]:=\frac{\mathrm{d}}{\mathrm{d}t} \Rcal_{x}(t\xi) \mid_{t=0} = \xi$, for all $\xi \in T_{x}\Mcal$. 
\end{itemize}
Note that the retraction operator may not be unique, e.g., the exponential maps \cite{absil2009optimization}, (certain types of) projections \cite{absil2012projection}.   
The following Lipschitz-type inequality of the retraction operator plays an important role in the analysis of retraction-based methods.
\begin{proposition}[\cite{boumal2019global}] \label{eq:lip-retr}
	Let $\mathcal{M}$ be a compact  submanifold of $\mathbb{R}^{d \times r}$. For all $x \in \mathcal{M}$ and $\xi \in T_{x} \mathcal{M}$, there exists a constant $M_1 >0$ such that the following inequality holds:
	\be \label{eq:bound-retraction1} \|\Rcal_{x}(\xi) - x - \xi \| \leq M_1 \|\xi \|^2, \; \forall x
	\in \Mcal, \; \forall \xi \in T_{x} \Mcal. \ee
\end{proposition}

Let $x_1,\cdots,x_n\in \Mcal$ represent the local copies of each agent. We denote $\hat{x}$ as the Euclidean average point of $x_1,\cdots,x_n$ given by
\bee
\hat{x}: = \frac{1}{n} \sum_{i=1}^n x_i.
\eee
Let $\Pcal_{\Mcal}$ be the orthogonal projection onto $\Mcal$. Note that for $\{x_i\}_{i=1}^n \subset \Mcal$,
\[ {\rm argmin}_{y\in\Mcal}\sum_{i=1}^n \|y - x_i\|^2 = \Pcal_{\Mcal}(\hat{x}). \]
Any element $\bar{x}$ in $\Pcal_{\Mcal}(\hat{x})$ is the induced arithmetic mean of $\{x_i\}_{i=1}^n$ on $\Mcal$ \cite{sarlette2009consensus}.
Let $f(z) : = \frac{1}{n}\sum_{i=1}^n f_i(z)$. The $\epsilon$-stationary point of problem \eqref{prob:original} is defined as follows.
\begin{definition} \label{def:station}
The set of points $\{x_1,x_2,\cdots,x_n\} \subset \Mcal$ is called an $\epsilon$-stationary point of \eqref{prob:original} if there exists a $\bar{x} \in \Pcal_{\Mcal}(\hat{x})$ such that
\[ \frac{1}{n}\sum_{i=1}^n \| x_i - \bar{x}\|^2 \leq \epsilon \quad {\rm and} \quad \|\grad f(\bar{x})\|^2 \leq \epsilon. \]
\end{definition}
In the following development, we always assure that $\hat{x} \in \bar{U}_{\Mcal}(\gamma)$. Consequently, $\Pcal_{\Mcal}(\hat{x})$ is a singleton and we have $\bar{x} = \Pcal_{\Mcal}(\hat{x})$. 

\section{Consensus problem over a compact submanifold}  \label{sec:linear-consensus}
To achieve the stationarity given in Definition \ref{def:station}, it is necessary to consider consensus on the compact submanifold $\Mcal$. However, the geodesic distance-based consensus problem \cite{tron2012riemannian} involves logarithm mapping and is challenging to solve. To address this, we consider the following consensus problem over $\Mcal$: 
\be \label{prob:consensus}
\min_{\bx} \phi^t(\bx): = \frac{1}{4} \sum_{i=1}^n \sum_{j=1}^n W^t_{ij}\|x_i - x_j\|^2,~ \text{s.t.}~x_i \in \Mcal, i\in [n].
\ee
The gradient of $\phi^t(\bx)$ is $\nabla \phi^t(\bx): = [\nabla \phi_1^t(\bx)^T, \nabla \phi_2^t(\bx)^T, \cdots, \nabla \phi_n^t(\bx)^T ]^T = (I_{nd} - \bW^t) \bx$, where $\nabla \phi_i^t(\bx): = x_i - \sum_{j=1}^n W_{ij}^tx_{j,k}, i \in [n]$. We use the following standard assumptions on $W$, see, e.g., \cite{zeng2018nonconvex,chen2021decentralized}.
Denote by the undirected agent network $G:=\{ \mathcal{V}, \mathcal{E}\}$, where $\mathcal{V}=\{1,2,\ldots, n\}$ is the set of all agents and $\mathcal{E}$ is the set of edges. Let $W$ be the adjacency matrix of $G$. Then $W_{ij} = W_{ji}$ and $W_{ij} > 0$ if an edge $(i,j) \in \mathcal{E}$ and otherwise $W_{ij} = 0$.
\begin{assumption} \label{assum-w}
     We assume that the undirected graph $G$ is connected and $W$ is doubly stochastic, i.e., (i) $W=W^{\top}$; (ii) $W_{i j} \geq 0$ and $1>W_{i i}>0$ for all $i, j$; (iii) Eigenvalues of $W$ lie in $(-1,1]$. The second largest singular value $\sigma_2$ of $W$ lies in $\sigma_2 \in[0,1)$.
\end{assumption}

Let us denote $\bar{\bx}_k = \mathbf{1}_n \otimes \bar{x}_k$ and $\hat{\bx}_k = \mathbf{1}_n \otimes \hat{x}_k$. For the Euclidean case (i.e., $\Mcal = \R^{d\times r}$), the gradient descent method with a unit step size has a locally linear convergence rate, in which the $k$-th iteration is given by $\bx_{k+1} = \bx_k - \nabla \phi^t(\bx_k)$. In fact, we have
\[ \begin{aligned}
    \| \bx_{k + 1} - \bar{\bx}_{k+1} \| & \leq \| \bx_{k+1} - \bar{\bx}_k \| = \| \bx_k - (I_{nd} - \bW^t)\bx_k - \bar{\bx}_k\| \\
    & = \| (\bW^t - J)(\bx_k - \bar{\bx}_k) \| \leq \sigma_2^t \|\bx_k - \bar{\bx}_k\|. 
\end{aligned} \]
When $\Mcal$ is the Stiefel manifold, the iterates generated by the Riemannian gradient descent method converge linearly under suitable step size and initialization, as shown in \cite[Theorem 2]{chen2021local}. However, it is not clear whether such linear convergence result holds for general compact submanifolds. Note that the linear convergence of (Riemannian) gradient methods is crucial in analyzing the convergence of decentralized gradient-type methods \cite{sun2020improving,chen2021decentralized}. This raises the question of whether we can design a gradient-type algorithm to solve \eqref{prob:consensus} with a locally linear convergence rate.

To this end, we consider the projected gradient method with a unit step size for solving \eqref{prob:consensus}, namely,
\be \label{eq:consensus-pg-iter}
x_{i,k+1} = \Pcal_{\Mcal}\left(\sum_{j=1}^n W_{ij}^tx_{j,k} \right),\quad i\in [n].
\ee
The main contribution of this section is to show the locally linear convergence of the scheme \eqref{eq:consensus-pg-iter} with an explicit characterization of the local neighborhood. The main technique utilized is the proximal smoothness of the compact submanifold $\Mcal$ \cite{clarke1995proximal}. Without loss of generality, we assume that $\Mcal$ is $2\gamma$-proximally smooth. By \eqref{eq:lip-proj} and \eqref{proximally-0}, the projection operator $\Pcal_{\Mcal}(\cdot)$ has the following properties:
  \begin{align}
\left\| \Pcal_{\mathcal{M}} (x) -\Pcal_{\mathcal{M}} (y)\right\| & \leq 2\|x - y\|,~~ \forall x,y \in \bar{U}_{\Mcal}(\gamma), \label{proximally-1} \\
 \left<v, y-x\right> &\leq \frac{\|v\|}{4\gamma}\|y-x\|^2,~ \forall x,y\in \Mcal, v\in N_x\Mcal. \label{proximally-2}
\end{align}
We will use the above inequality to characterize the locally linear convergence of the projected gradient method and the associated neighborhood.

Let us define a neighborhood around $\bx\in \Mcal^n$ and a constant $\zeta$ as follows:
\begin{equation}\label{def:N-neiborhood}
    \mathcal{N}:=\{ \bx\in \Mcal^n: \max_i\|x_i - \bar{x} \| \leq \gamma/2 \},\;\; \zeta: = \max_{x, y\in\Mcal}\|x - y\|.
\end{equation}
For a more general iterate scheme, $x_{i,k+1} = \Pcal_{\Mcal}\left(\sum_{j=1}^n W_{ij}^tx_{i,k} - \alpha_k u_{i,k} \right)$, where $\alpha_k > 0$ and $u_{i,k} \in \R^{d\times r}$, the following lemma demonstrates that the iterates ${\bx_{k+1}}$ will remain in the neighborhood $\mathcal{N}$, provided that $t$, $\alpha_k$, and $u_{i,k}$ satisfy certain conditions and $\bx_k \in \mathcal{N}$.


\begin{lemma} \label{lem:stay-neighborhood}
Let $x_{i,k+1} = \Pcal_{\Mcal}\left(\sum_{j=1}^n W_{ij}^tx_{i,k} - \alpha_k  u_{i,k} \right)$. Suppose that Assumption \ref{assum-w} holds.
If $\bx_k \in \mathcal{N}$, $\| u_{i,k} \| \leq B$,  $\alpha_k \leq \gamma/(24B)$, and $t \geq  \left\lceil \log_{\sigma_2}\left(\frac{\gamma}{24\sqrt{n}\zeta }\right) \right \rceil$ with $\zeta$ given in \eqref{def:N-neiborhood}, then it holds that $\bx_{k+1}\in \mathcal{N}$ and 
\begin{align}
    \sum_{j=1}^n {W_{ij}^t} x_{j,k} -\alpha_k u_{i,k} & \in \bar{U}_{\Mcal} (\gamma), ~ i=1,\cdots,n. \label{eq:neibohood1}
\end{align}
\end{lemma}
\begin{proof}
Note that for any $i\in [n]$,
\[ \begin{aligned}
&\| \sum_{j=1}^n W^t_{ij} x_{j,k} + \alpha_k u_{i,k} - \bar{x}_k  \|
\leq    \| \sum_{j=1}^n W^t_{ij} x_{j,k} + \alpha_k u_{i,k} - \hat{x}_k  \| + \| \hat{x}_k - \bar{x}_k \| \\
& \leq  \| \sum_{j=1}^n (W^t_{ij} - \frac{1}{n})(x_{j,k} - \hat{x}_k) \| + \alpha_k B +\max_i\| \hat{x}_k - x_{i,k} \|\\
& \leq    \sum_{j=1}^n \left|W^t_{ij} - \frac{1}{n}\right|\|x_{j,k} - \hat{x}_k\| + \alpha_k B + \frac{1}{2} \gamma \\
& \leq  \zeta \max_i  \sum_{j=1}^n \left|W^t_{ij} - \frac{1}{n}\right|+ \alpha_k B + \frac{1}{2} \gamma \\
& \leq \sqrt{n} \sigma_2^t \zeta + \alpha_k B + \frac{1}{2} \gamma
\leq \gamma,
\end{aligned}\]
where the second inequality utilizes $\|\hat{x}_k - \bar{x}\| = \|\hat{x}_k - \Pcal_{\Mcal}(\hat{x}_k)\| \leq \|\hat{x}_k - x_{i,k}\|$ for all $i\in [n]$, the fourth inequality uses the fact that $\|x_{j,k} - \hat{x}_k\| \leq \frac{1}{n}\sum_i\|x_{j,k} - x_{i,k}\| \leq \zeta$, the fifth inequality follows from the bound on the total variation distance between any row of $W^t$ and $\frac{1}{n}\mathbf{1}$ \cite{diaconis1991geometric,boyd2004fastest}. Combining the fact that $\bar{x}_k \in \Mcal$, we obtain \eqref{eq:neibohood1}. By the definition of $\hat{x}_{k}$ and $\bx_k \in \mathcal{N}$, we have that for any $i\in [n]$,
\begin{align*}
    & \| x_{i,k+1} - \bar{x}_{k+1} \|
   \leq \|  x_{i,k+1} - \bar{x}_k \| + \|\bar{x}_k - \bar{x}_{k+1}\| \\
   \leq & \|  x_{i,k+1} - \bar{x}_k \| + 2\|\hat{x}_{k+1} - \bar{x}_k \| \\
 \leq  & 3\max_i \| x_{i,k+1} - \bar{x}_{k} \| \\
  = &3 \max_i \left\| \Pcal_{\Mcal}( \sum_{j=1}^n W^t_{ij}x_{j,k} - \alpha_k u_{i,k} ) -
\Pcal_{\Mcal}( \hat{x}_k)  \right\|
 \\ 
 \overset{\eqref{proximally-1}}{\leq} & 6 \max_{i} \left \| \sum_{j=1}^n W^t_{ij}x_{j,k}- \alpha_k  u_{i,k} - \hat{x}_k \right\| 
 \leq  6 \sqrt{n}  \sigma_2^t \zeta  +6 \alpha_k B  \leq \frac{1}{2} \gamma,
\end{align*}
where the second inequality utilizes that $\|\Pcal_{\Mcal}(x) - y\| \leq \|\Pcal_{\Mcal}(x) - x\| + \|x - y\|\leq 2 \|x-y\|$ for any $x,y\in \mathbb{R}^{d\times r}$. This implies that $\bx_{k+1}\in \mathcal{N}$.  The proof is completed.
\end{proof}

Lemma \ref{lem:stay-neighborhood} demonstrates that under certain conditions on $\alpha_k, t$ and $u_{i,k}$, if $x_0 \in \mathcal{N}$, then for all $k$, it holds that $\bx_k\in \mathcal{N}$ and $\sum_{j=1}^n W_{ij}^tx_{i,k} - \alpha_k  u_{i,k}$ remains within the neighborhood $\bar{U}_{\Mcal}(\gamma)$. This latter result allows us to invoke the Lipschitz continuity \eqref{proximally-1} of $\Pcal_{\Mcal}$ over $\bar{U}_{\Mcal}(\gamma)$ in the subsequent analysis. Moreover, Lemma \ref{lem:stay-neighborhood} is established based on a general iterative scheme with bounded $u_{i,k}$. When $u_{i,k} = 0$, this reduces to \eqref{eq:consensus-pg-iter}. As will be seen in Section \ref{sec:main-alg}, Lemma \ref{lem:stay-neighborhood} is useful in analyzing the decentralized optimization algorithms, where $u_{i,k}$ is an approximate gradient.

Let us denote $\Pcal_{\Mcal^n}(\bx) = [\Pcal_{\Mcal}(x_1)^\top, \ldots, \Pcal_{\Mcal}(x_n)^\top]^\top$. Based on Lemma \ref{lem:stay-neighborhood}, we show that the projected gradient method \eqref{eq:consensus-pg-iter} converges to the consensus set at a locally linear rate. 
\begin{theorem}[Linear convergence of consensus error] \label{theo:linear-con-consensus}
   Let $\{\bx_k\}$ be the iterate sequences generated by \eqref{eq:consensus-pg-iter}.  Suppose that Assumption \ref{assum-w} holds. If $\bx_{0} \in \mathcal{N}$,  $t > \max\left\{ \left\lceil \log_{\sigma_2}\left(\frac{\gamma}{24\sqrt{n}\zeta}\right) \right \rceil, \left\lceil \log_{\sigma_2}(1/2) \right \rceil \right\}$, then, $\bx_k \in \mathcal{N}$ for all $k\geq 0 $ and the following linear convergence with  rate $\rho_t: = 2\sigma_2^t < 1$ holds,
   \bee
    \|\bx_{k+1} - \bar{\bx}_{k+1}\| \leq  \rho_t \|\bx_k -\bar{\bx}_k\|.
   \eee
   
\end{theorem}

\begin{proof}
Since $\|\hat{x}_0 - \bar{x}_0\| \leq \frac{1}{2}\gamma$, by invoking Lemma \ref{lem:stay-neighborhood} with $\alpha_k = 0, u_{i,k} = 0, i\in [n]$, it follows that for any $k>0$, $\bx_k \in \mathcal{N}$ and
\begin{equation}\nonumber
    \begin{aligned}
       &  \sum_{j=1}^n {W_{ij}^t} x_{j,k}  \in \bar{U}_{\Mcal} (\gamma), ~ i\in [n].
    \end{aligned}
\end{equation}
     Based on the iterative scheme in \eqref{eq:consensus-pg-iter}, we have
\be\label{eq:wx-hatx} \begin{aligned}
    \|\bx_{k+1} - \bar{\bx}_{k+1}\| & \leq \| \bx_{k+1} - \bar{\bx}_k \|  = \| \Pcal_{\Mcal^n}(\bW^t \bx_k) - \Pcal_{\Mcal^n}( \hat{\bx}_k) \| \\
    & \leq 2 \|\bW^t \bx_k -\hat{\bx}_k  \|  = 2 \| (W^t \otimes I_d) \bx_k - \hat{\bx}_k \| \\
    & = 2 \| ((W^t -J ) \otimes I_d) (\bx_k - \hat{\bx}_k) \|  \\
    & \leq 2 \sigma_2^t \|\bx_k -\hat{\bx}_k\| \leq  2 \sigma_2^t \|\bx_k -\bar{\bx}_k\|,
\end{aligned}
\ee
where the first inequality utilizes \eqref{proximally-1}.  The proof is completed.
\end{proof}

The result above is also applicable to the case when $\Mcal$ is a convex set. In this scenario, by utilizing the 1-Lipschitz continuity of the projection operator, the linear rate can be improved to $\sigma_2^t$, which is consistent with the results in \cite{nedic2010constrained,nedic2018network}. It is worth noting that the rate in Theorem \ref{theo:linear-con-consensus} will asymptotically go to $\sigma_2^t$ as ${\Pcal}_{\Mcal}$ is approximately $1$-Lipschitz continuously near $\Mcal$ (i.e., the second inequality in \eqref{eq:wx-hatx} can be tighter when $\bW^t \bx_k$ and $\hat{\bx}_k$ are closer to $\Mcal^n$).   

\section{Useful inequalities about compact submanifolds} \label{sec:ineq-man}
In this section, we will present several useful inequalities related to the compact submanifold, which will serve as key components in the analysis of decentralized algorithms for solving \eqref{prob:original} in the next section. We believe that such inequalities will be useful in tackling manifold optimization problems of a more general nature. Let us start with the following assumption concerning problem \eqref{prob:original}.
\begin{assumption}\label{assum-f}
    Each objective function $f_i$ is of gradient Lipschitz continuous with modulus $L_{f}$ on the convex hull of $\Mcal$, denoted by ${\rm conv}(\Mcal)$, i.e., for any $x, y \in {\rm conv}(\Mcal)$, it holds that
\be \label{eq:egrad-lip}
\| \nabla f_i(x) - \nabla f_i(y)  \| \leq L_f \|x- y\|,\quad i\in [n].
\ee
Moreover, the Euclidean gradient is bounded by $L_G$, i.e., $ \max_{x \in \Mcal} \|  \nabla f_i(x)\| \leq L_G, \; i\in [n].$
\end{assumption}

The above assumption is standard and commonly used in decentralized optimization \cite{nocedal2006numerical,zeng2018nonconvex,chen2021decentralized}. Given the compactness of $\Mcal$, Assumption \ref{assum-f} holds for any differentiable function $f$ with a locally Lipschitz continuous gradient. Using \eqref{eq:egrad-lip}, we can readily obtain a quadratic upper bound for $f_i$ as follows:
\be\label{eq:grad-Lip}
f_i(y) \leq f_i(x) + \left<\nabla f_i(x), y-x \right> + \frac{L_f}{2} \| y - x\|^2, \;\; \forall x, y \in {\rm conv}(\Mcal), \;\; i\in [n].
\ee

Based on \eqref{eq:grad-Lip} and the properties of proximally smooth sets, we show Riemannian quadratic upper bound for $f_i$ in the following lemma. 
\begin{lemma}[Riemannian quadratic upper bound]\label{lemma:lipsctz}
Under Assumption \ref{assum-f}, for any $x,y\in \Mcal$, the following inequality holds:
\begin{equation}\label{f-riemannian-Lip}
    f_i(y) - f_i(x) \leq \left<\grad f_i(x),y-x\right>+\frac{L_g}{2}\|y-x\|^2, ~ i\in [n],
\end{equation}
where
$
L_g: = L_f + \frac{1}{2\gamma} L_G.
$
Moreover, we have
\begin{equation}\label{g-riemannian-lip}
    \| \grad f_i(x) - \grad f_i(y) \| \leq (L_f + L_GL_{\Pcal}) \|x - y\|,~i\in [n],
\end{equation}
where $L_{\Pcal}$ is a positive constant.
\end{lemma}

\begin{proof}
    Under Assumption \ref{assum-f}, for any $i\in [n]$, it holds that
    \begin{equation}
        \begin{aligned}
          &   f_i(y) - f_i(x) - \left<\grad f_i(x),y-x\right> \\
          = & f_i(y) - f_i(x) - \left<\nabla f_i(x),y-x\right> + \left<\Pcal_{N_x\Mcal} (\nabla f_i(x)),y-x\right> \\
             \leq & \frac{L_f}{2}\|y-x \|^2 + \frac{ \left\| \Pcal_{N_x\Mcal} (\nabla f_i(x))  \right\|}{4\gamma}\|y - x\|^2 \\
             \leq &\left( \frac{L_f}{2} + \frac{1}{4\gamma} \max_{z\in \Mcal}\|\nabla f_i(z)\|  \right) \| y - x\|^2.
        \end{aligned}
    \end{equation}
   where the first inequality utilizes \eqref{eq:grad-Lip} and \eqref{proximally-2}. This implies \eqref{f-riemannian-Lip}. It follows from \eqref{eq:grad-Lip} that 
   \be
    \begin{aligned}
     & \| \grad f_i(x) - \grad f_i(y) \| = \| \Pcal_{T_x\Mcal}(\nabla f_i(x)) - \Pcal_{T_y\Mcal}(\nabla f_i(y)) \| \\
     \leq  & \| \Pcal_{T_x\Mcal}(\nabla f_i(x) - \nabla f_i(y))  \| + \| \Pcal_{T_x\Mcal}(\nabla f_i(y)) - \Pcal_{T_y\Mcal}(\nabla f_i(y)) \| \\
     \leq & L_f \|y-x \| + \| \nabla f_i(y)\| L_{\mathcal{P}}(\|y-x\|) \\
     \leq & (L_f + L_G L_{\Pcal}) \|y-x \|, \\
    \end{aligned}
   \ee
   where we use the Lipschitz continuity of $\Pcal_{T_{x} \Mcal}$ over $x \in \Mcal$ ($L_{\Pcal}$ is the associate modulus). This gives \eqref{g-riemannian-lip}. The proof is completed.
\end{proof}

The above Riemannian quadratic upper bound will serve as a key tool for analyzing the convergence of Riemannian gradient-type algorithms, paralleling its counterpart in the Euclidean space.

Analogous to the role of Lipschitz-type inequalities of the retraction for Riemannian gradient algorithms \cite{boumal2019global},  the following Lipschitz-type inequality for the projection operator $\Pcal_{\Mcal}(\cdot)$ is crucial in the analysis of projection-based methods.

\begin{lemma}[Lipschitz-type inequalities on the projection operator $\Pcal_{\Mcal}(\cdot)$]\label{lemma-project}
For any $x\in\Mcal, u\in \{u\in \mathbb{R}^{d\times r}:\|u\|\leq \gamma\}$, 
there exists a constant $Q$ such that
\be\label{projec-second-order1}
\| \Pcal_{\Mcal}(x + u)  - x - \Pcal_{T_x\Mcal}(u) \| \leq Q\|u\|^2.
\ee
\end{lemma}
\begin{proof}
Let us denote $u_1 = \Pcal_{T_x\Mcal}(u)$, $u_2 = u - u_1$.  Since the projection operator is a retraction operator \cite{absil2012projection}, it follows from Proposition \ref{eq:lip-retr} that
 \begin{equation}\label{eq:second-order}
 \begin{aligned}
     \|\Pcal_{\Mcal}(x + u_1) - x - u_1\| \leq M_1 \|u_1\|^2 \leq M_1 \|u\|^2.
   \end{aligned}
\end{equation}
For a $2\gamma$-proximally smooth $\Mcal$, by \eqref{proximally-2}, it holds $\iprod{u}{y-x} \leq \frac{\|u\|}{4\gamma}\|y-x\|^2$ for any $x,y \in \Mcal$ and $u \in N_{x}\Mcal$. Then 
\[ \|y-(x + u)\|^2 = \|y-x\|^2 - 2\iprod{u}{y-x} + \|u\|^2 \geq (1 - \frac{\|u\|}{2\gamma})\|y-x\|^2 + \|u\|^2 \geq \|u\|^2. \]
Hence, for any $x\in \Mcal$ and $\|u\|\leq 2 \gamma$,
\be \label{eq:proj-x}  x = \arg\min_{y\in \Mcal} \|y - (x+u)\|^2 = \Pcal_{\Mcal}(x+u). \ee
Let $K = \{\eta \in \R^{d\times r} : \|\eta\| \leq \gamma \}$ be a compact subset of the tangent bundle. For all $u \in K$,  by noting the smoothness of $\Pcal_{\Mcal}$ \cite[Lemma]{foote1984regularity}, we have
\be
\begin{aligned}
 & \| \Pcal_{\Mcal}(x + u) - \Pcal_{\Mcal}(x + u_1)  \|  \leq \int_{0}^1 \left\| \frac{d}{dt}(  \Pcal_{\Mcal}(x + u_1 + t u_2) )  \right\| dt  \\
= & \int_{0}^1 \left\|   D\Pcal_{\Mcal}(x + u_1 + t u_2)[u_2]  \right\| dt \\
 \leq & \int_{0}^1 \left\|   D\Pcal_{\Mcal}(x + u_1 + t u_2)[u_2]  - D\Pcal_{\Mcal}(x  + t u_2)[u_2] \right\| dt \\
\leq &  \max_{\eta \in K} \|D^2 \Pcal_{\Mcal}(x + \eta) \|_{\rm op} \|u_1\|\|u_2\|,
\end{aligned}
\ee
where $\|\cdot\|_{\rm op}$ represents the operator norm and the first inequality uses $D\Pcal_{\Mcal}(x  + t u_2)[u_2] = \lim_{\epsilon \rightarrow 0} (\Pcal_{\Mcal}(x  + t u_2 + \epsilon u_2) - \Pcal_{\Mcal}(x  + t u_2))/\epsilon = 0$ by \eqref{eq:proj-x}. For all $u \notin K$, we have
\be\label{projec-second-order2}
\| \Pcal_{\Mcal}(x + u)  - \Pcal_{\Mcal}(x + u_1) \| \leq \zeta \leq \frac{\zeta }{\gamma^2} \|u\|^2,
\ee
where $\zeta = \max_{x,y \in \Mcal}\|x-y\|$ is the maximal distance between any two points on $\Mcal$. 
Defining
\[
Q_0 = \max \left(\max_{\eta \in K} \|D^2 \Pcal_{\Mcal}(x + \eta) \|_{\rm op},   \frac{\zeta }{\gamma^2} \right)
\]
and combining \eqref{projec-second-order2} with  \eqref{eq:second-order}, we have the following result:
\be
\begin{aligned}
\| \Pcal_{\Mcal}(x + u)  - x - u_1 \| & \leq \| \Pcal_{\Mcal}(x + u)  - \Pcal_{\Mcal}( x + u_1) \| + \| \Pcal_{\Mcal}(x + u_1)  - x - u_1 \| \\
&\leq Q_0 \|u\|^2 + M_1\|u_1\|^2 \leq (Q_0+M_1)\|u\|^2,
\end{aligned}
\ee
which gives \eqref{projec-second-order1} with $Q = Q_0 + M_1$.
\end{proof}

Another useful inequality that we establish is the control of the distance between the Euclidean mean $\hat{x}$ and the manifold mean $\bar{x}$ by the square of consensus error.
\begin{lemma}\label{lemma:quadratic}
   For any $\bx\in \Mcal^n$ satisfying $\|x_i - \bar{x}\| \leq \gamma$,  we have 
   \be\label{eq:distance-an-rm}
   \|\bar{x} - \hat{x} \| \leq M_2 \frac{\|\bx - \bar{\bx}\|^2}{n},
   \ee
   where $M_2 =\max_{x\in {\rm conv}(\Mcal)} \|D^2 \Pcal_{\Mcal}(x) \|_{\rm op}$.
\end{lemma}

\begin{proof}
Since $x_i \in \Mcal$,  the definitions of $\hat{x}$ and $\bar{x}$ yield
\be
\begin{aligned}
\hat{x} & = \frac{1}{n} \sum_{i=1}^n x_i = \frac{1}{n} \sum_{i=1}^n \Pcal_{\Mcal}(x_i),~\bar{x}  = \frac{1}{n} \sum_{i=1}^n \Pcal_{\Mcal}(\bar{x}).
\end{aligned}
\ee
Since $\bar{x} = \Pcal_{\Mcal}(\hat{x}) $, we get $\hat{x} - \bar{x} \in N_{\bar{x}}\Mcal.$  Therefore,
\be
\frac{1}{n}\sum_{i=1}^n D\Pcal_{\Mcal}(\bar{x})[x_i - \bar{x} ]  =  D\Pcal_{\Mcal}(\bar{x})\left[\frac{1}{n}\sum_{i=1}^n x_i - \bar{x}  \right] = D\Pcal_{\Mcal}(\bar{x})[\hat{x} - \bar{x}  ] = 0,
\ee
where we use the fact in \cite[Lemma 3.1]{absil2012projection} that $D \Pcal_\Mcal( \bar{x}) = \Pcal_{T_{\bar{x}\Mcal}}$. 
Then we have 
\begin{equation}
    \begin{aligned}
    \left\|\hat{x} - \bar{x}    \right \| & =     \left\| \frac{1}{n} \sum_{i=1}^n\left( \Pcal_{\Mcal}(x_i)- \Pcal_{\Mcal}(\bar{x}) \right)   \right \|  = \int_{0}^1 \left\|\frac{1}{n} \sum_{i=1}^n \frac{d}{dt}\left(  \Pcal_{\Mcal}(\bar{x} +  t (x_i - \bar{x} )) \right)  \right\| dt  \\
    &= \int_{0}^1 \left\|  \frac{1}{n} \sum_{i=1}^n D\Pcal_{\Mcal}(\bar{x} +  t ( x_i - \bar{x} ))[ x_i - \bar{x} ]  \right\| dt \\
    &= \int_{0}^1 \left\|  \frac{1}{n} \sum_{i=1}^n \left( D\Pcal_{\Mcal}(\bar{x} +  t ( x_i - \bar{x} ))[ x_i - \bar{x} ]  - D\Pcal_{\Mcal}(\bar{x})[ x_i - \bar{x} ] \right) \right\| dt \\
    & \leq  \frac{1}{n} \sum_{i=1}^n \max_{x\in \text{conv}(\Mcal)} \|D^2 \Pcal_{\Mcal}(x) \|_{\rm op} \left\| x_i - \bar{x} \right\|^2  \\
    & =\max_{x\in \text{conv}(\Mcal)} \|D^2 \Pcal_{\Mcal}(x) \|_{\rm op} \frac{\|\bx - \bar{\bx}\|^2}{n},
    \end{aligned}
\end{equation}
where we use the smoothness of $\Pcal_{\Mcal}$ over $\bar{U}_{\Mcal}(\gamma)$ \cite[Lemma]{foote1984regularity}.  The proof is completed.
\end{proof}

\section{Decentralized projected Riemannian gradient-type methods}\label{sec:main-alg}
In this section, we present two decentralized projected Riemannian gradient-type methods, the decentralized projected Riemannian gradient method (DPRGD) and the decentralized projected Riemannian gradient tracking method (DPRGT) for solving problem \eqref{prob:original}, and corresponding convergence analysis.  
\subsection{The Algorithms}
In the DPRGD method for solving \eqref{prob:original}, each step consists of a projected gradient step for the consensus and a Riemannian gradient step for the local objective function $f_i$.
Specifically, given an adjacency matrix $W$ of the communication network, in the $k$-th iteration, the DPRGD performs the following update
\begin{equation}\label{eq:dpg}
\begin{aligned}
   x_{i,k+1} = \Pcal_{\Mcal}\left(\sum_{j=1}^n W_{ij}^tx_{j,k} - \alpha_k  \grad f_i(x_{i,k}) \right),\;\; i \in [n],   
\end{aligned}
\end{equation}
where $\alpha_k > 0$ is the step size and $t \geq 1$ is an integer. The DPRGD method is presented in Algorithm \ref{alg:drpgd}. 

\begin{algorithm}[htbp]
\caption{Decentralized Projected Riemannian Gradient Descent (DPRGD) for solving \eqref{prob:original}} \label{alg:drpgd}
\begin{algorithmic}[1]
\REQUIRE  Initial point $\bx_0\in \mathcal{N}$, an integer $t$, set $k = 1$.
\WHILE{the condition is not met}
\STATE Choose diminishing step size $\alpha_k = \mathcal{O}(1/\sqrt{k})$.
\STATE Update $$x_{i,k+1} = \Pcal_{\Mcal}\left(\sum_{j=1}^n W_{ij}^tx_{j,k} - \alpha_k  \grad f_i(x_{i,k}) \right),$$ for each node $i\in [n]$, in parallel.
\STATE Set $k=k+1$.
\ENDWHILE
\end{algorithmic}
\end{algorithm}

As demonstrated in Section \ref{sec:linear-consensus}, our projected gradient step with a unit step size, i.e., $\Pcal_{\Mcal}(\sum_{j=1}^n W_{ij}^t x_{jk})$, for the consensus problem \eqref{prob:consensus} achieves locally linear convergence, which enables us to establish the convergence of Algorithm \ref{alg:drpgd} without introducing an additional step size parameter on the consensus.

\begin{algorithm}[htbp]
\caption{Decentralized projected Riemannian gradient tracking method for solving \eqref{prob:original}} \label{alg:drgta}
\begin{algorithmic}[1]
\REQUIRE  Initial point $\bx_0\in \mathcal{N}$, an integer $t$, the step size $\alpha$. Set $k = 0$.
\STATE Let $s_{i,0} = \grad f_i(x_{i,0})$ on each node $i\in [n]$.
\WHILE {the condition is not met}
\STATE Project onto tangent space: $v_{i,k} = \Pcal_{T_{x_{i,k}}\Mcal}( s_{i,k})$.
\STATE Update $$x_{i,k+1} = \Pcal_{\Mcal}(\sum_{j=1}^n W_{ij}^tx_{j,k} - \alpha
v_{i,k}), ~i\in [n].$$
\STATE Riemannian gradient tracking:
\be
s_{i,k+1}  = \sum_{j=1}^n W_{ij}^ts_{j,k} + \grad f_i(x_{i,k+1}) - \grad f_i(x_{i,k}), ~i\in [n].
\ee
\STATE Set $k=k+1$.
\ENDWHILE
\end{algorithmic}
\end{algorithm}
    

Next, we investigate a gradient tracking method for solving \eqref{prob:original} by leveraging the gradient tracking techniques introduced in \cite{nedic2017achieving,qu2017harnessing,chen2021decentralized}  to get a better estimate for the full gradient. 
In the $k$-th iteration, our DPRGT method performs the following update, for all $i \in [n]$, 
  \bee
      \begin{aligned}
        x_{i,k+1} & = \Pcal_{\Mcal}\left(\sum_{j=1}^n W_{ij}^tx_{j,k} - \alpha
        \Pcal_{T_{x_{i,k}}\Mcal}( s_{i,k}) \right),  \\
        s_{i,k+1}  & = \sum_{j=1}^n W_{ij}^ts_{j,k} + \grad f_i(x_{i,k+1}) - \grad f_i(x_{i,k}),
      \end{aligned}
  \eee
  where $\alpha > 0$ is the step size and the projection of $s_{i,k}$ to $T_{x_{i,k}} \Mcal$ is used in $x_{i,k+1}$. The detailed description is presented in Algorithm \ref{alg:drgta}. A crucial advantage of gradient tracking-type methods lies in the applicability of the use of a constant step size $\alpha$.

\subsection{Convergence analysis}
This subsection focuses on the complexity results for DPRGD and DPRGT algorithms. Let us first present the following main theorem on the $\mathcal{O}(\epsilon^{-2})$ iteration complexity of the DPRGD method (i.e., Algorithm \ref{alg:drpgd}) to reach the $\epsilon$-stationary point of \eqref{prob:original}. For the ease of analysis, we define $L:=\max\{L_G, L_g, \\ L_f + L_GL_{\Pcal}\}$.
\begin{theorem} \label{thm:dprgd}
     Let $\{\bx_k\}_k$ be the sequence generated by Algorithm \ref{alg:drpgd}. Suppose that Assumptions \ref{assum-w} and \ref{assum-f} hold. If $\bx_0 \in \mathcal{N}$, $\alpha_k  = \frac{1}{\sqrt{k+1}} \min\{\gamma/(24L),1 \},$ and $ t \geq \max\left\{\lceil\log_{\sigma_2}(1/2) \rceil, \left\lceil \log_{\sigma_2}\left(\frac{\gamma}{24\sqrt{n}\zeta}\right) \right \rceil \right\}$,  
 it follows that:
  \be\label{main-consensus}\frac{1}{n}\|\bar{\bx}_K - \bx_K \|^2 \leq \mathcal{O}(\frac{1}{K}),\ee 
     \be\label{main-grad}
  \min_{k \le K}\|\grad f(\bar{x}_k) \|^2 = \mathcal{O}( \frac{1}{\sqrt{K+1}}).
     \ee
\end{theorem}

We would like to highlight that it is possible to develop a stochastic variant of Algorithm \ref{alg:drpgd} by employing a stochastic estimation of the gradient $\grad f_i(x_{i,k})$, whose convergence properties can be shown analogously.

The following theorem shows that by incorporating gradient tracking, Algorithm \ref{alg:drgta} yields improved complexity in reaching an $\epsilon$-stationary point.
\begin{theorem} \label{thm:dprgt}
  Let $\{\bx_k\}_k$ be the sequence generated by Algorithm \ref{alg:drgta}. Suppose that Assumptions \ref{assum-w} and \ref{assum-f} hold. If $\bx_0 \in \mathcal{N}$,
  \be
  t\geq \max\left\{\log_{\sigma_2}(\frac{1}{4\sqrt{n}}),\log_{\sigma_2}(\frac{\gamma}{24\sqrt{n}\zeta}) \right\},
  \ee
  and  $$\alpha<\min \left\{ \frac{\gamma}{32 L}, \frac{1}{24L},\frac{1}{L^2}, 1, \frac{1}{4\left(\tilde{C}_2(128\tilde{C}_0  + 8)L^2 + 2(\mathcal{C}_1 + 4\tilde{C}_0(\mathcal{C}_2 + \mathcal{C}_2)) \right) }\right\},$$
  where $\tilde{C}_0,\tilde{C}_2,\mathcal{C}_1, \mathcal{C}_2$ are given in the proof in the Subsection \ref{eq:proof1},  it follows that
\bee
\begin{aligned}
    \min_{k \leq K} \frac{1}{n}\|\bs_k\|^2  = \mathcal{O}(\frac{1}{\alpha K}), ~
    \min_{k\leq K} \frac{1}{n} \|\bx_k - \bar{\bx}_k \|^2   = \mathcal{O}(\frac{1}{K}),~
    \min_{k\leq K} \|\grad f(\bar{x}_k) \|^2 =\mathcal{O}(\frac{1}{\alpha K}).
\end{aligned}
\eee
\end{theorem}

The proofs of Theorems \ref{thm:dprgd} and \ref{thm:dprgt} are shown in Subsections  \ref{sec:proof:dprgd} and \ref{sec:proof:dprgt}, respectively.

\subsubsection{Two key lemmas}
In the following lemma, we show that $\left\| \sum_{i=1}^n \grad \phi_i^t(\bx) \right \|$ is bounded by the square of consensus error. 
\begin{lemma} \label{lem:grad-consensus}
Let $L_2 := 2 \max_{x \in {\rm conv}(\Mcal)} \| D^2 \Pcal_{T_x\Mcal}(\cdot) \|_{\rm op}$. For any $\bx \in \Mcal^n$,  it holds that
\be\label{eq:sum-consen-grad} \| \sum_{i=1}^n \grad \phi_i^t(\bx)  \| \leq \sqrt{n} L_2 \|\bx - \bar{\bx}\|^2.  \ee
\end{lemma}
\begin{proof}
    Using the fact that $\sum_{i=1}^n \nabla \phi_i^t(\bx) = \sum_{i=1}^n \left[  x_i - \sum_{j=1}^n W_{ij}^t x_j \right] = 0$, we have
    \[ \begin{aligned}
       & \left\| \sum_{i=1}^n \grad \phi_i^t(\bx) \right \|  = \left \|  \sum_{i=1}^n \left[ \Pcal_{T_{x_i}\Mcal}(\nabla \phi_i^t(\bx)) - \Pcal_{T_{\bar{x}}\Mcal}(\nabla \phi_i^t(\bx)) \right]   \right \| \\
        & \leq \frac{L_2}{2} \sum_{i=1}^n \left[ \| x_i - \bar{x}\| \|\nabla \phi_i^t (\bx)\| \right] 
         \leq \frac{L_2}{2}  \left( \max_{i\in [n]} \|x_i - \bar{x} \|\right) \sum_{i=1}^n \| x_i - \sum_{j=1} W_{ij}^tx_j \| \\
        & \leq \frac{L_2}{2}  \left( \max_{i\in [n]} \|x_i - \bar{x} \|\right) \left[ \sum_{i=1}^n \| x_i -\bar{x}\| + \| \sum_{j=1} W_{ij}^tx_j -\bar{x} \| \right]\\
        & \leq L_2 \left( \max_{i\in [n]} \|x_i - \bar{x} \|\right) \sum_{i=1}^n \| x_i - \bar{x} \| 
         \leq \sqrt{n} L_2  \|\bx - \bar{\bx}\|^2,
    \end{aligned}\]
    where the first inequality is due to the Lipschitz continuity of $\Pcal_{T_{x}\Mcal}(\cdot)$ over $x \in \Mcal$, the third inequality is due to the triangle inequality of $\|\cdot \|$, the fourth inequality uses the convexity of $\|\cdot\|$, and the last inequality comes from $\|a\|_1 \leq \sqrt{n}\|a\|$ and $\|a\|_{\infty} \leq \|a\|$ for any $a \in \R^{n}$. 
\end{proof}

The next technical result bounds the distance between $\bar{x}_{k+1}$ and $\bar{x}_{k}$ for a given iterative process.
\begin{lemma} \label{lem:diff-avg-x}
Let $x_{i,k+1} = \Pcal_{\Mcal}\left(\sum_{j=1}^n W_{ij}^tx_{i,k} - \alpha_k  u_{i,k} \right)$, where $u_{i,k} \in T_{x_{i,k}}\Mcal$. Denote $\hat{u}_k: = \frac{1}{n}\sum_{i=1}^n u_{i,k}$. Suppose that Assumption \ref{assum-w} holds. 
It holds that
\be \label{eq:diff-avg-x}
\begin{aligned}
& \| \bar{x}_{k+1} - \bar{x}_k \| \\
\leq & \frac{8Q + \sqrt{n}L_2 + M_2}{n} \|\bx_k - \bar{\bx}_k\|^2 +\frac{2Q\alpha_k^2}{n} \| \bu_k \|^2  + \alpha_k \| \hat{u}_{k}  \|  + \frac{M_2}{n} \| \bx_{k+1}  - \bar{\bx}_{k+1} \|^2. \end{aligned} \ee
\end{lemma}
\begin{proof}
Since $\| \nabla \phi^t(\bx_k)\|  = \|(I_{nd} - \bW^t) \bx_k \| = \|(I_{nd} - \bW^t) (\bx_k - \bar{\bx}_k) \| \leq 2 \|\bx_k - \bar{\bx}_k \|$, we have
\be
\begin{aligned}
    & \| \hat{x}_{k+1} - \hat{x}_k \| \\
    \leq & \| \hat{x}_{k+1} - \hat{x}_k + \frac{1}{n} \sum_{i=1}^n (\grad \phi_i^t(\bx_k) + \alpha_k u_{i,k}  ) \| + \|\frac{1}{n} \sum_{i=1}^n (\grad \phi_i^t(\bx_k) + \alpha_k u_{i,k}  )\| \\
    \overset{\eqref{projec-second-order1} }{\leq} & \frac{Q}{n}\sum_{i=1}^n \| \nabla \phi_i^t(\bx_k) + \alpha_k u_{i,k} \|^2 +  \| \frac{1}{n} \sum_{i=1}^n\grad \phi_i^t(\bx_k)\|  + \alpha_k \| \hat{u}_{k}  \| \\
    \leq &\frac{2Q}{n}\| \nabla \phi^t(\bx_k)\|^2 + \frac{2Q\alpha_k^2}{n} \| \bu_k \|^2 +  \| \frac{1}{n} \sum_{i=1}^n\grad \phi_i^t(\bx_k)\|  + \alpha_k \| \hat{u}_{k}  \| \\
   \overset{\eqref{eq:sum-consen-grad}}{ \leq} & \frac{8Q + \sqrt{n}L_2}{n} \|\bx_k - \bar{\bx}_k\|^2 +\frac{2Q\alpha_k^2}{n} \| \bu_k \|^2  + \alpha_k \| \hat{u}_{k}  \|.
\end{aligned}
\ee
Therefore, we have
\begin{equation}
    \begin{aligned}
 &   \| \bar{x}_{k+1} - \bar{x}_k \|  \leq \| \hat{x}_{k+1} - \hat{x}_k  \| + \| \hat{x}_{k+1} - \bar{x}_{k+1}  \| + \| \hat{x}_{k} - \bar{x}_{k}  \| \\
    \overset{\eqref{eq:distance-an-rm}}{ \leq} &   \frac{8Q + \sqrt{n}L_2}{n} \|\bx_k - \bar{\bx}_k\|^2 +\frac{2Q\alpha_k^2}{n} \| \bu_k \|^2  + \alpha_k \| \hat{u}_{k}  \|  \\
    &+ \frac{M_2}{n} (  \| \bx_k  - \bar{\bx}_k \|^2 + \| \bx_{k+1}  - \bar{\bx}_{k+1} \|^2 ) \\
     \leq & \frac{8Q + \sqrt{n}L_2 + M_2}{n} \|\bx_k - \bar{\bx}_k\|^2 +\frac{2Q\alpha_k^2}{n} \| \bu_k \|^2  + \alpha_k \| \hat{u}_{k}  \|  \\
     &+ \frac{M_2}{n} \| \bx_{k+1}  - \bar{\bx}_{k+1} \|^2.
    \end{aligned}
\end{equation}
The proof is completed.
\end{proof}

\subsubsection{Proof of Theorem \ref{thm:dprgd}.}\label{sec:proof:dprgd}

Let us start with some notations. Denote $\grad f(\bx_k) =[\grad f_1(x_{1,k})^\top, \ldots, \\ \grad f_n(x_{n,k})^\top]^\top$ and ${\bf G}_k :=[\grad f_1(x_{1,k})^\top, \cdots, \grad f_n(x_{n,k})^\top ]^\top$. 
By appropriately selecting the step size $\alpha_k$ and integer $t$, and with an initialization $\bx_0 \in \mathcal{N}:= \{ \bx: \|\bar{x} - \hat{x}\| \leq \gamma/2 \}$, we show in the following lemma on the consensus error based on Lemma \ref{lem:stay-neighborhood}. 
\begin{lemma}\label{lemma:dpg:consensus}
Let $\{\bx_k\}_k$ be the sequence generated by Algorithm \ref{alg:drpgd}. Suppose that Assumptions \ref{assum-w} and \ref{assum-f} hold. If $\bx_0 \in \mathcal{N}$, $\| \grad f_i(x_{i,k}) \| \leq L$,  $\alpha_k \leq \gamma/(24L)$, and $t \geq  \left\lceil \log_{\sigma_2}\left(\frac{\gamma}{24\sqrt{n}\zeta}\right) \right \rceil$,  it follows that $\bx_k \in \mathcal{N}$ for all $k \geq 0$ and
\begin{equation}
\begin{aligned}
 \|\bx_{k+1} - \bar{\bx}_{k+1}\| & \leq 2\sigma_2^t \|\bx_{k} - \bar{\bx}_{k}\| + 2\sqrt{n}\alpha_k L .
\end{aligned}
\end{equation}
\end{lemma}
\begin{proof}
Since $\|\hat{x}_0 - \bar{x}_0\| \leq \frac{1}{2}\gamma$ and $\|\grad f_i(x_{i,k}) \| \leq L$, it follows from Lemma \ref{lem:stay-neighborhood} that for any $k>0$, the following holds:
\begin{equation}\label{eq:neibohood3}
    \begin{aligned}
       &  \sum_{j=1}^n {W_{ij}^t} x_{j,k} -\alpha_k \grad f_i(x_{i,k})  \in \bar{U}_{\Mcal} (\gamma), ~ i = 1,\cdots,n.
    \end{aligned}
\end{equation}
By the definition of $\bar{\bx}_{k+1}$, we have
\[ \begin{aligned}
    \|\bx_{k+1} - \bar{\bx}_{k+1}\| & \leq \| \bx_{k+1} - \bar{\bx}_k \| \\
    & = \| \Pcal_{\Mcal^n}(\bW^t \bx_k - \alpha_k  \grad f(\bx_k)) - \Pcal_{\Mcal^n}( \hat{\bx}_k) \| \\
    & \leq 2 \|\bW^t \bx_k - \alpha_k  \grad f(\bx_k) -\hat{\bx}_k  \| \\
    & \leq 2\sigma_2^t \|\bx_k -\bar{\bx}_k\| + 2   \sqrt{n} \alpha_k L,
\end{aligned}
\]
where the first inequality follows from the optimality of $\bar{\bx}_{k+1}$, the second inequality utilizes \eqref{eq:neibohood3} and the $2$-Lispchitz continuity of $\Pcal$ over $\bar{U}_{\Mcal}(\gamma)$, the last inequality utilizes \eqref{eq:wx-hatx}. We complete the proof.  
\end{proof}

With the above lemma, we can elaborate on a more explicit relationship between the consensus error and the step size.  
\begin{lemma} \label{lem:consensus}
    Let $\{\bx_k\}$ be the sequence generated by Algorithm \ref{alg:drpgd} and $\rho_t = 2\sigma_2^t$. Suppose that Assumptions \ref{assum-w} and \ref{assum-f} hold. If $\bx_0 \in \mathcal{N}$, $\alpha_k = \min\{ \gamma/(24L), \frac{1}{\sqrt{k+1}} \}$ and $ t \geq \max\left\{\log_{\sigma_2}(1/2), \left\lceil \log_{\sigma_2}\left(\frac{\gamma}{24\sqrt{n}\zeta}\right) \right \rceil \right\}$, 
    then $\bx_k \in \mathcal{N}$ for all $k\geq 0$ and
    there exists a constant $C>0$ such that 
    $$\frac{1}{n}\|\bar{\bx}_k - \bx_k \|^2 \leq CL^2 \alpha_k^2,$$ 
    where $C$ is independent of $L$ and $n$.
\end{lemma}

\begin{proof}
    It follows from Lemma \ref{lemma:dpg:consensus} that
    \be \label{eq:temp:pgd:consensus}
    \begin{aligned}
        \|\bar{\bx}_{k+1} - \bx_{k+1} \| & \leq \rho_t \|\bar{\bx}_k - \bx_k\| + 2\sqrt{n}\alpha_k L \\
        & \leq \rho_t^{k+1} \|  \bar{\bx}_0 - \bx_0\| + 2\sqrt{n}L \sum_{l=0}^k \rho_t^{k-l}\alpha_l.
    \end{aligned}
    \ee
   Let $ a_k: = \frac{\|\bar{\bx}_k - \bx_k\|}{\sqrt{n} \alpha_k}$. For a given positive integer number $K\leq k$,  it follows from \eqref{eq:temp:pgd:consensus} that
   \bee
   a_{k+1} \leq \rho_t a_k + 2L \frac{\alpha_k} 
 {\alpha_{k+1}} \leq \rho_t^{k+1 - K}a_{K} + 2L \sum_{l=K}^k \rho_t^{k-l} \frac{\alpha_l}{\alpha_{l+1}}.
   \eee
   Since that $\alpha_k = \mathcal{O}(1/L)$ and $\|\bar{\bx}_0 - \bx_0 \| \leq \frac{1}{2} \sqrt{n}  \gamma$, one have that $a_0 \leq \frac{1}{2}\gamma/\alpha_0 = \mathcal{O}(L)$. Due to that $\lim_{k\rightarrow \infty} \frac{\alpha_{k+1}}{\alpha_k} = 1$, there exists sufficiently large $K$ such that $\alpha_k / \alpha_{k+1} \leq 2,~ \forall k\geq K$.  For $0\leq k \leq K$, there exists $C^{'} > 0$ such that $a_k^2 \leq C^{'} L^2$, where $C^{'}$ is independent of $L$ and $n$. For $k\geq K$, one has that $a_k^2 \leq CL^2$, where $C = 2C^{'} + \frac{32}{(1-\rho_t)^2}$. Hence, we get $\frac{\|\bar{\bx}_k - \bx_k\|^2}{n} \leq CL^2 \alpha_k^2$ for all $k\geq 0$, where $C = \mathcal{O}(\frac{1}{(1-\rho_t)^2})$.      
\end{proof}

By utilizing the Lipschitz-type inequalities on compact submanifolds in Section \ref{sec:ineq-man} and combining the above lemma, we can show a sufficient decrease on $f$. 
\begin{lemma}\label{them:dpg}
Let $\{\bx_k\}$ be the sequence generated by Algorithm \ref{alg:drpgd}. Suppose that Assumptions \ref{assum-w} and \ref{assum-f} hold. If $\bx_0 \in \mathcal{N}$, $ t \geq \max \left\{\log_{\sigma_2}(1/2), \left\lceil \log_{\sigma_2}\left(\frac{\gamma}{24\sqrt{n}\zeta}\right) \right \rceil \right \}$, and  $\alpha_k =  \frac{1}{\sqrt{k+1}} \min\{\gamma/(24L),1 \}$, 
it follows that
\bee
\begin{aligned}
       f(\bar{x}_{k+1})  \leq & f(\bar{x}_k) - \frac{\alpha_k}{4} \|\grad f(\bar{x}_k) \|^2 +
  \mathcal{G}_1   \alpha_k^3 + \mathcal{G}_2  \alpha_k^4,
\end{aligned}
\eee
where
\bee
\begin{aligned}
    \mathcal{G}_1 & = (C  + 6C^2(M_2^2+(\sqrt{n}L_2+8Q)^2)     + 24Q^2)L^4,\\
    \mathcal{G}_2 & = ( 2C^2(8Q + \sqrt{n}L_2 + M_2)^2 + 2C^2M_2^2 + 8Q^2  )L^5.
\end{aligned}
\eee
\end{lemma}

\begin{proof}
It follows from the Riemannian quadratic upper bound of $f$ in Lemma \ref{lemma:lipsctz} and $L_g \leq L$ that
\be
\label{eq:dgd:major:inequality}
\begin{aligned}
& f(\bar{x}_{k+1})  \leq f(\bar{x}_k) + \left<\grad f(\bar{x}_k),   \bar{x}_{k+1} - \bar{x}_{k}\right> + \frac{L}{2}\|\bar{x}_{k+1} - \bar{x}_{k} \|^2 \\
 = & f(\bar{x}_k) - \left<\grad f(\bar{x}_k), \alpha_k \hat{g}_k \right> + \left<\grad f(\bar{x}_k),  \bar{x}_{k+1} - \bar{x}_{k} + \alpha_k \hat{g}_k \right> + \frac{L}{2}\|\bar{x}_{k+1} - \bar{x}_{k} \|^2 \\
= & f(\bar{x}_k) - \frac{\alpha_k}{2}\|\grad f(\bar{x}_k) \|^2 - \frac{\alpha_k}{2}\| \hat{g}_k \|^2 + \frac{\alpha_k}{2}  \| \grad f(\bar{x}_k) - \hat{g}_k \|^2 \\
  & +   \left<\grad f(\bar{x}_k),  \bar{x}_{k+1} - \bar{x}_{k} + \alpha_k \hat{g}_k \right> + \frac{L}{2} \|\bar{x}_{k+1} - \bar{x}_{k} \|^2.
\end{aligned}
\ee
According to Young's inequality, we have
\be \label{eq:grad-prod}
\left<\grad f(\bar{x}_k),  \bar{x}_{k+1} - \bar{x}_{k} + \alpha_k \hat{g}_k \right> \leq \frac{\alpha_k}{4}\|\grad f(\bar{x}_k) \|^2 + \frac{1}{\alpha_k}\| \bar{x}_{k+1} - \bar{x}_{k} + \alpha_k \hat{g}_k\|^2.
\ee
Combining \eqref{eq:dgd:major:inequality} and  \eqref{eq:grad-prod} leads to
\be\label{eq:dgd:major:inequality1}
\begin{aligned}
 f(\bar{x}_{k+1})  \leq &  f(\bar{x}_k) - \frac{\alpha_k}{4}\|\grad f(\bar{x}_k) \|^2 - \frac{\alpha_k}{2}\| \hat{g}_k \|^2 + \frac{\alpha_k}{2} \underbrace{ \| \grad f(\bar{x}_k) - \hat{g}_k \|^2}_{a_1} \\
  & +  \frac{1}{\alpha_k} \underbrace{\| \bar{x}_{k+1} - \bar{x}_{k} + \alpha_k \hat{g}_k\|^2}_{a_2} + \frac{L}{2} \underbrace{\|\bar{x}_{k+1} - \bar{x}_{k} \|^2.}_{a_3}
\end{aligned}
\ee
Now, let us bound $a_1$, $a_2$, and $a_3$, respectively. Applying Lemma \ref{lemma:lipsctz} yields
\bee
\begin{aligned}
a_1 & =  \left\|\grad f(\bar{x}_k) - \frac{1}{n}\sum_{i=1}^n \grad f_i(x_{i,k}) \right \|^2  \leq  \frac{1}{n}\sum_{i=1}^n\|\grad f_i(\bar{x}_k) -  \grad f_i(x_{i,k}) \|^2 \\
& \leq    \frac{L^2}{n} \sum_{i=1}^n \|\bar{x}_k -  x_{i,k} \|^2 = \frac{L^2}{n} \| \bar{\bx}_k - \bx_k \|^2.
\end{aligned}
\eee
For $a_2$, it follows from the triangle inequality that
\be\label{eq:tmp2}
\begin{aligned}
    & \| \bar{x}_{k+1} - \bar{x}_{k} + \alpha_k \hat{g}_k\| 
   \leq  \| \bar{x}_k - \hat{x}_k \| + \| \bar{x}_{k+1} - \hat{x}_{k+1} \| + \| \hat{x}_{k+1} - \hat{x}_{k} + \alpha_k \hat{g}_k\| \\
   \overset{\eqref{eq:distance-an-rm}}{\leq} &  \frac{M_2}{n}(\| \bar{\bx}_k - \bx_k \|^2 + \| \bar{\bx}_{k+1} - \bx_{k+1} \|^2) + \| \hat{x}_{k+1} - \hat{x}_{k} + \alpha_k \hat{g}_k\|.
\end{aligned}
\ee
Moreover, we have
\be\label{eq:tmp1}
\begin{aligned}
  &  \| \hat{x}_{k+1} - \hat{x}_{k} + \alpha_k \hat{g}_k\| 
  =  \| \frac{1}{n}\sum_{i=1}^n (x_{i,k+1} - x_{i,k} + \alpha_k \grad f_i(x_{i,k}) ) \|  \\
  \leq  &\frac{1}{n} \|  \sum_{i=1}^n (x_{i,k+1} - x_{i,k} + \alpha_k \grad f_i(x_{i,k}) + \grad \phi_i^t( \bx_k )) \| + \frac{1}{n}\|  \sum_{i=1}^n \grad \phi_i^t( \bx_k ) \| \\
  \overset{\eqref{projec-second-order1} }{\leq} & \frac{Q}{n} \sum_{i=1}^n  \| \alpha_k \grad f_i(x_{i,k}) + \nabla \phi_i^t( \bx_k )  \|^2 + \frac{1}{n}\|  \sum_{i=1}^n \grad \phi_i^t( \bx_k ) \| \\
  \overset{\eqref{eq:sum-consen-grad}}{\leq} & \frac{2Q\alpha_k^2}{n} \|{\bf G}_k \|^2 + \frac{2Q}{n} \|\nabla \phi^t( \bx_k ) \|^2  +\frac{L_2}{\sqrt{n}} \| \bx_k - \bar{\bx}_k \|^2 \\
  \leq & \frac{2 Q\alpha_k^2}{n} \|{\bf G}_k \|^2 + \frac{(\sqrt{n}L_2 + 8Q)}{n}\| \bx_k - \bar{\bx}_k \|^2.
\end{aligned}
\ee
Plugging \eqref{eq:tmp1} into \eqref{eq:tmp2} gives
\bee
\begin{aligned}
    a_2  \leq & \frac{3M_2^2+6(\sqrt{n}L_2+8Q)^2}{n^2}\| \bx_k - \bar{\bx}_k \|^4  \\
    &+ \frac{3M_2^2}{n^2}\| \bar{\bx}_{k+1} - \bx_{k+1} \|^4 + \frac{24Q^2 \alpha_k^4}{n^2} \|{\bf G}_k \|^4.
\end{aligned}
\eee
By \eqref{eq:diff-avg-x}, we obtain
\begin{equation}\nonumber
    \begin{aligned}
    a_3  \leq & \frac{4(8Q + \sqrt{n}L_2 + M_2)^2}{n^2} \|\bx_k - \bar{\bx}_k\|^4 +\frac{16Q^2\alpha_k^4}{n^2} \| {\bf G}_k \|^4  + 4\alpha_k^2 \| \hat{g}_{k}  \|^2   \\
    & + \frac{4M_2^2}{n^2} \| \bx_{k+1}  - \bar{\bx}_{k+1} \|^4.
    \end{aligned}
\end{equation}
 Combining $a_1, a_2, a_3$ with \eqref{eq:dgd:major:inequality1} and noting $\|{\bf G}_k\|^2 \leq n L^2$ imply that
\bee
\begin{aligned}
 f(\bar{x}_{k+1})  \leq &  f(\bar{x}_k) - \frac{\alpha_k}{4}\|\grad f(\bar{x}_k) \|^2 - \frac{\alpha_k}{2}\| \hat{g}_k \|^2 + \frac{\alpha_k}{2} a_1  +  \frac{1}{\alpha_k} a_2 + \frac{L}{2} a_3\\
 \leq & f(\bar{x}_k) - \frac{\alpha_k}{4}\|\grad f(\bar{x}_k) \|^2 - ( \frac{\alpha_k}{2} - 2L\alpha_k^2  )\|  \hat{g}_k\|^2 + \frac{L^2\alpha_k}{2n} \| \bar{\bx}_k - \bx_k \|^2  \\
 & +(\frac{3M_2^2+6(\sqrt{n}L_2+8Q)^2}{n^2\alpha_k}+\frac{2L(8Q + \sqrt{n}L_2 + M_2)^2}{n^2} ) \| \bx_k - \bar{\bx}_k \|^4 \\
 & +( \frac{3M_2^2}{n^2\alpha_k} +  \frac{2LM_2^2}{n^2} )\| \bar{\bx}_{k+1} - \bx_{k+1} \|^4 + (24Q^2\alpha_k^3 + 8Q^2L  \alpha_k^4) L^4
\end{aligned}
\eee
Since $\alpha_{k+1} \leq \alpha_k \leq 1/(4L)$ and $\frac{1}{n}\|\bar{\bx}_k - \bx_k \|^2 \leq CL^2 \alpha_k^2 $ by Lemma \ref{lem:consensus}, it holds that
\bee
\begin{aligned}
    f(\bar{x}_{k+1})  \leq & f(\bar{x}_k) - \frac{\alpha_k}{4} \|\grad f(\bar{x}_k) \|^2 \\
    & + (C  + 6C^2(M_2^2+(\sqrt{n}L_2+8Q)^2)     + 24Q^2)L^4\alpha_{k}^3 \\
    & + ( 2C^2(8Q + \sqrt{n}L_2 + M_2)^2 + 2C^2M_2^2 + 8Q^2  )L^5\alpha_{k}^4.
\end{aligned}
\eee
The proof is completed.
\end{proof}

With these preparations, we give the proof of Theorem \ref{thm:dprgd}.

\begin{proof}[Proof of Theorem \ref{thm:dprgd}] Firstly, \eqref{main-consensus} follows from Lemma \ref{lem:consensus}. Now we prove \eqref{main-grad}. By Lemma \ref{them:dpg}, we have
    \bee
\begin{aligned}
       f(\bar{x}_{k+1})  \leq & f(\bar{x}_k) - \frac{\alpha_k}{4}  \|\grad f(\bar{x}_k) \|^2  +
 \mathcal{G}_1 \alpha_k^3 +\mathcal{G}_2 \alpha_k^4.
\end{aligned}
\eee
Summing the above inequality over $k=0,1,\ldots, K$ and telescoping the right-hand side with any $K>0$ give
\bee
\sum_{k=0}^K \frac{\alpha_k}{4}  \|\grad f(\bar{x}_k) \|^2 \leq f(\bar{x}_{0})  - f(\bar{x}_*) +  \mathcal{G}_1 \sum_{k=0}^K \alpha_k^3 + \mathcal{G}_2\sum_{k=0}^K \alpha_k^4.
\eee
Dividing both sides by $\sum_{k=0}^K \frac{\alpha_k}{4}$ yields
\bee
\min_{k=0,\cdots,K}\|\grad f(\bar{x}_k) \|^2 \leq \frac{f(\bar{x}_{0})  - f(\bar{x}_*) +  \mathcal{G}_1 \sum_{k=0}^K \alpha_k^3 + \mathcal{G}_2\sum_{k=0}^K \alpha_k^4}{\sum_{k=0}^K \frac{\alpha_k}{4}}.
\eee
Since $\alpha_k = \mathcal{O}(\frac{1}{\sqrt{k}})$, it holds that $\frac{\sum_{k=0}^K \alpha_k^3 }{\sum_{k=0}^K \alpha_k } = \mathcal{O}( \frac{1}{\sqrt{K+1}})$ dominates the right hand of the above inequality.  The proof is completed.
\end{proof}

\subsubsection{Proof of Theorem \ref{thm:dprgt}}\label{sec:proof:dprgt}
For ease of notation, let us denote 
$$
\begin{aligned}
\hat{g}_k &:=\frac{1}{n} \sum_{i=1}^n \grad f_i\left(x_{i, k}\right), \quad \hat{\mathbf{G}}_k:=\left(\mathbf{1}_n \otimes I_d\right) \hat{g}_k, \quad  \\
\bv_k & :=[v_{1,k}^\top, \cdots, v_{n,k}^\top]^\top,\quad \bs_{k}:=[s_{1,k}^\top, \cdots, s_{n,k}^\top]^\top. 
\end{aligned}
$$
Firstly, we have the following lemma on the relationship between ${\bf s}_k$ and $\hat{\bf G}_k$.
\begin{lemma} \label{lem:diff-s-g}
Let $\{\bx_k\}_k$ be the sequence generated by Algorithm \ref{alg:drgta}. It holds that for any $k$,
\be \label{eq:diff-g-g} \|{\bf G}_{k+1} - {\bf G}_k\| \leq 4L\|\bx_k - \bar{\bx}_k\| + 2L\alpha \| \bv_k\|,
\ee
\be \label{eq:diff-s-g} \|{\bf s}_{k+1} - \hat{\bf G}_{k+1} \| \leq \sigma_2^t \| {\bf s}_k - \hat{\bf G}_k \| + \| {\bf G}_{k+1} - {\bf G}_k \|. \ee

\end{lemma}
\begin{proof}
The Lipschitz continuity of ${\bf G}$ yields
\begin{equation}\nonumber
\begin{aligned}
      \|{\bf G}_{k+1} - {\bf G}_k\| &  \leq L \|\bx_{k+1} - \bx_k\|.
\end{aligned}
\end{equation}
On the other hand, it holds that
\begin{equation}\nonumber
    \begin{aligned}
        \|  \bx_{k+1} - \bx_k \| & \leq \|  \bx_{k+1} - \bx_k + (I_{ nd}-\bW^t)\bx_k + \alpha \bv_k \| + \| (I_{ nd}-\bW^t)\bx_k + \alpha \bv_k\| \\
        & \leq 2\| (I_{ nd}-\bW^t)\bx_k + \alpha \bv_k \| \leq 4\|\bx_k - \bar{\bx}_k\| + 2\alpha \| \bv_k\|,
    \end{aligned}
\end{equation}
where the first inequality is from the triangle inequality and the second inequality is due to the definition of $\bx_{k+1}$. 
It follows from the definition of $ \hat{\bf G}_{k+1}$ that
\[ \begin{aligned}
{\bf s}_{k+1} - \hat{\bf G}_{k+1} & = ((I_n - J) \otimes I_d) {\bf s}_{k+1} \\
& = ((I_n -J) \otimes I_d) ( (W \otimes I_d) {\bf s}_k + {\bf G}_{k+1} - {\bf G}_{k})  \\
& = ((W^t - J) \otimes I_d) {\bf s}_k + ((I_n - J) \otimes I_d)({\bf G}_{k+1} - {\bf G}_k).\\
\end{aligned} \]
Using the spectral property of $W$, we conclude that \eqref{eq:diff-s-g} holds.
\end{proof}

Before showing the boundedness of consensus error, we investigate the uniform boundedness of $\|\bs_k\|$ in the following lemma.
\begin{lemma}\label{lemma:neibohood}
Let $\{\bx_k\}_k$ be the sequence generated by Algorithm \ref{alg:drgta}. Suppose that Assumptions \ref{assum-w} and \ref{assum-f} hold. If $\bx_0 \in \mathcal{N}$, $t\geq \max\left\{\log_{\sigma_2}(\frac{1}{4\sqrt{n}}), \log_{\sigma_2}\left(\frac{\gamma}{24\sqrt{n}\zeta}\right)\right\}$, and $\alpha< \frac{\gamma}{96 L}$, it follows that for all $k$, $\bx_k \in \mathcal{N}$  and
\be\label{eq:sk-bound}
\|s_{i,k}\| \leq 4 L,~ \forall i\in [n].
\ee
\end{lemma}

\begin{proof}
We prove it by induction on both $\|s_{i,k}\|$ and $ \max_i\| x_{i,k} - \bar{x}_{k} \|$. Since $\|s_{i,0}\| = \| \grad f_i(x_{i,0}) \| \leq L $ for all $i\in [n]$ and $\max_i\|x_{i,0} - \bar{x}_0\| \leq \frac{1}{2}\gamma$.  Suppose for some $k\geq 0$ that $\|s_{i,k}\| \leq 4 L$ and $\max_i\|x_{i,k} - \bar{x}_k\| \leq \frac{1}{2}\gamma$.  Since $\|v_{i,k}\| \leq \|s_{i,k}\| \leq 4 L$ and $\alpha< \gamma/(96L)$, it follows from Lemma \ref{lem:stay-neighborhood} that
\bee
    \begin{aligned}
         \sum_{j=1}^n {W_{ij}^t} x_{j,k} -\alpha  v_{i,k} \in \bar{U}_{\Mcal} (\gamma), ~ i = 1,\cdots,n, \quad \max_i\| x_{i,k+1} - \bar{x}_{k+1} \| \leq \frac{1}{2} \gamma.
    \end{aligned}
\eee 
Then, we have
\bee
\begin{aligned}
\| s_{i,k+1} - \hat{g}_k\| & = \| \sum_{j=1}^n W^t_{ij}s_{j,k} - \hat{g}_k + \grad f_i(x_{i,k+1}) - \grad f_i(x_{i,k}) \| \\
& =  \| \sum_{j=1}^n (W^t_{ij} - \frac{1}{n})  s_{j,k} \|  + \|\grad f_i(x_{i,k+1}) - \grad f_i(x_{i,k}) \| \\
& \leq \sigma_2^t \sqrt{n} \max_i \| s_{i,k}\| + 2L  \\
& \leq  \sigma_2^t \sqrt{n} \max_i\| s_{i,k}\| + 2 L \\
& \leq \frac{1}{4} \max_i\| s_{i,k}\| + 2 L \leq L + 2 L \leq 3L. 
\end{aligned}
\eee
Hence, $
\| s_{i,k+1}   \| \leq\| s_{i,k+1} - \hat{g}_k\| + \| \hat{g}_k\| \leq 3L + L \leq 4\zeta L$, where we use $\| \hat{g}_k\| \leq \frac{1}{n} \sum_{i=1}^n\| \grad f_i(x_{i,k}) \|\leq  L$. The proof is completed.
\end{proof}

With the above lemma, we have the following result on the consensus error.
\begin{lemma}\label{lemma:dpgta:consensus}
Let $\{\bx_k\}$ be the sequence generated by Algorithm \ref{alg:drgta} and $\rho_t = 2\sigma_2^t$. Under the same conditions in Lemma \ref{lemma:neibohood}, the following holds
\begin{equation}\label{eq:dgta:consensus}
\begin{aligned}
 \|\bx_{k+1} - \bar{\bx}_{k+1}\| & \leq 2\sigma_2^t \|\bx_{k} - \bar{\bx}_{k}\| + 2\alpha\|\bv_k\|.
\end{aligned}
\end{equation}
 Moreover, there exists a constant $D$ such that
$
\frac{1}{n}\| \bx_k - \bar{\bx}_k \|^2 \leq  D L^2 \alpha^2.
$
\end{lemma}

\begin{proof}
    Similar to the proof in Theorem \ref{theo:linear-con-consensus}, we have
    \[ \begin{aligned}
    \|\bx_{k+1} - \bar{\bx}_{k+1}\| & \leq \| \bx_{k+1} - \bar{\bx}_k \|  = \| \Pcal_{\Mcal^n}(\bW^t \bx_k - \alpha  \bv_k) - \Pcal_{\Mcal^n}( \hat{\bx}_k) \| \\
    & { \leq }2 \|\bW^t \bx_k - \alpha  \bv_k -\hat{\bx}_k  \| 
    \leq 2\sigma_2^t \|\bx_k -\bar{\bx}_k\| + 2\alpha\|\bv_k\|.
\end{aligned}
\]
Since $\|\bv_k\| \leq \|\bs_k\| \leq 4\sqrt{n}L$ by \eqref{eq:sk-bound}, using the same 
argument of Lemma \ref{lemma:dpg:consensus}, there exists $D = \mathcal{O}(\frac{1}{(1-\rho_t)^2})$, which is independent of $L$ and $n$, such that
$
\frac{1}{n} \| \bx_{k} - \bar{\bx}_k \|^2 \leq  D L^2 \alpha^2,~\forall k\geq 0.
$
\end{proof}

By utilizing the Lipschitz-type inequalities on compact submanifolds in Section \ref{sec:ineq-man} and combining the above lemma, we can show a sufficient decrease on $f$. 

\begin{lemma}\label{lemma:dgta:majorized}
Let $\{\bx_k\}_k$ be the sequence generated by Algorithm \ref{alg:drgta}. Suppose that Assumptions \ref{assum-w} and \ref{assum-f} hold. If $\bx_0 \in \mathcal{N}$, $t\geq \max\left\{\log_{\sigma_2}(\frac{1}{4\sqrt{n}}), \log_{\sigma_2}\left(\frac{\gamma}{24\sqrt{n}\zeta}\right)\right\}$, and $\alpha< \frac{\gamma}{192 L}$, it follows that
\begin{equation}\nonumber
    \begin{aligned}
&      f(\bar{x}_{k+1}) \\
\leq &  f(\bar{x}_k) - (\alpha  - \frac{L^2}{2} \alpha^2) \|\hat{g}_k\|^2 + \mathcal{C}_1  \alpha^2 \frac{1}{n}\|\bs_k \|^2 +\mathcal{C}_2 \frac{1}{n} \| \bx_k - \bar{\bx}_k\|^2 +\mathcal{C}_3 \frac{1}{n} \| \bx_{k+1} - \bar{\bx}_{k+1}\|^2,
\end{aligned}
\end{equation}
where
\begin{equation}\nonumber
    \begin{aligned}
          \mathcal{C}_1 & = (2Q+ 3)L + 192 Q^2\alpha^2L^3 + 3L, \\
         \mathcal{C}_2 & = 4LQ+4L + M_2^2DL +  4DL^3(8Q + \sqrt{n}L_2 + M_2)^2\alpha^2, \\
         \mathcal{C}_3 & = M_2^2D L + 4DL^3M_2^2\alpha^2.
    \end{aligned}
\end{equation}
\end{lemma}

\begin{proof}
It follows from Lemma \ref{lemma:lipsctz} and $L_g \leq L$ that 
\bee
\| \hat{g}_k - \grad f(\bar{x}_k) \|^2 \leq \frac{1}{n} \sum_{i=1}^n \|\grad f_i(x_k) -  \grad f_i(\bar{x}_k) \|^2 \leq \frac{L^2}{n} \|\bx_k - \bar{\bx}_k \|^2.
\eee
and
\begin{equation}\label{Lip-ineq}
    \begin{aligned}
    f(\bar{x}_{k+1})  \leq & f(\bar{x}_k) + \left<\grad f(\bar{x}_k),\bar{x}_{k+1}-\bar{x}_k  \right>+ \frac{L}{2}\|\bar{x}_{k+1}-\bar{x}_k\|^2\\
     = & f(\bar{x}_k) + \left<\hat{g}_k ,\bar{x}_{k+1}-\bar{x}_k  \right>+\left<\grad f(\bar{x}_k) -\hat{g}_k,\bar{x}_{k+1}-\bar{x}_k  \right> + \frac{L}{2}\|\bar{x}_{k+1}-\bar{x}_k\|^2 \\
      \leq & f(\bar{x}_k) + \left<\hat{g}_k ,\bar{x}_{k+1}-\bar{x}_k  \right>+ \frac{3L}{4}\|\bar{x}_{k+1}-\bar{x}_k\|^2  + \frac{1}{L} \| \grad f(\bar{x}_k) -\hat{g}_k \|^2 \\
        \leq & f(\bar{x}_k) + \left<\hat{g}_k ,\hat{x}_{k+1}-\hat{x}_k  \right>  + \left<\hat{g}_k,  \bar{x}_{k+1}- \hat{x}_{k+1}+\hat{x}_k - \bar{x}_k  \right> \\
        &+ \frac{L}{n}  \| \bx_k - \bar{\bx}_k \|^2 + \frac{3L}{4}\|\bar{x}_{k+1}-\bar{x}_k\|^2.
    \end{aligned}
\end{equation}
By Young's inequality, we get 
\begin{equation} \label{eq:grad-prod2}
    \left<\hat{g}_k,  \bar{x}_{k+1}- \hat{x}_{k+1}+\hat{x}_k - \bar{x}_k  \right> \leq \frac{\alpha^2L}{2}\| \hat{g}_k  \|^2 + \frac{1}{\alpha^2L}(\| \bar{x}_{k+1}- \hat{x}_{k+1} \|^2 + \|\hat{x}_k - \bar{x}_k\|^2).
\end{equation}
Combining \eqref{Lip-ineq} and \eqref{eq:grad-prod2} leads to
\begin{equation}\nonumber
    \begin{aligned}
     f(\bar{x}_{k+1})  \leq & f(\bar{x}_k) + \underbrace{\left<\hat{g}_k ,\hat{x}_{k+1}-\hat{x}_k  \right>}_{b_1}  +  \frac{\alpha^2L}{2}\| \hat{g}_k   \|^2 + \underbrace{ \frac{1}{\alpha^2L}(\| \bar{x}_{k+1}- \hat{x}_{k+1} \|^2 + \|\hat{x}_k - \bar{x}_k\|^2)}_{b_2}\\
     & + \frac{L}{n}  \| \bx_k - \bar{\bx}_k \|^2 + \underbrace{ \frac{3L}{4}\|\bar{x}_{k+1}-\bar{x}_k\|^2}_{b_3}. \\
    \end{aligned}
\end{equation}
Now, let us bound $b_1$, $b_2$, and $b_3$, respectively. For $b_1$, it holds that
\begin{equation} \label{eq:b1-0}
    \begin{aligned}
    b_1  =& \left< \hat{g}_k , \frac{1}{n} \sum_{i=1}^n( x_{i,k+1} - x_{i,k} + \alpha s_{i,k} + \nabla \phi_i^t(\bx_k)   ) \right> -  \left< \hat{g}_k, \frac{1}{n} \sum_{i=1}^n(  \alpha s_{i,k} + \nabla \phi_i^t(\bx_k)   ) \right> \\
    \leq &  \left< \hat{g}_k   , \frac{1}{n} \sum_{i=1}^n( x_{i,k+1} - x_{i,k} + \alpha v_{i,k} + \nabla \phi_i^t(\bx_k)  ) \right> +\left< \hat{g}_k, \frac{1}{n} \sum_{i=1}^n \alpha(s_{i,k} - v_{i,k}) \right> - \alpha  \|\hat{g}_k\|^2.
    \end{aligned}
\end{equation}
Since $s_{i,k} - v_{i,k} \in N_{x_{i,k}}\Mcal$, it follows that
\be \label{eq:b1-1}
\begin{aligned}
    & \left< \hat{g}_k, \frac{1}{n} \sum_{i=1}^n \alpha(s_{i,k} - v_{i,k}) \right>
   = \frac{\alpha}{n} \sum_{i=1}^n\left< \hat{g}_k - \grad f_i(x_{i,k}),  s_{i,k} - v_{i,k} \right> \\
   \leq & \frac{1}{4nL} \sum_{i=1}^n \|\hat{g}_k - \grad f_i(x_{i,k})  \|^2 + \frac{\alpha^2L}{n}\sum_{i=1}^n \| \Pcal_{N_{x_{i,k}}\Mcal}(s_{i,k}) \|^2\\
   \leq & \frac{1}{4n^2L} \sum_{i=1}^n \sum_{j=1}^n\|\grad f_j(x_{j,k}) - \grad f_i(x_{i,k})  \|^2 + \frac{\alpha^2L}{n}\| \bs_k \|^2 \\
   &
   \leq  \frac{L}{n} \| \bx_k - \bar{\bx}_k \|^2 + \frac{\alpha^2L}{n}\| \bs_k \|^2.
\end{aligned}
\ee
Let  $\nabla \phi_i^t(\bx_k) + \alpha v_{i,k} = d_{i,1} + d_{i,2}, $
where $d_{i,1} = \Pcal_{T_{x_{i,k}}\Mcal}(\nabla \phi_i^t(\bx_k) + \alpha v_{i,k}) , d_{i,2} = \nabla \phi_i^t(\bx_k)+ \alpha v_{i,k} - d_{i,1}$.  Therefore,
\be \label{eq:b1-2}
\begin{aligned}
&\left< \hat{g}_k   , \frac{1}{n} \sum_{i=1}^n( x_{i,k+1} - x_{i,k} + \alpha v_{i,k} + \nabla \phi_i^t(\bx_k)  ) \right> \\
= & \frac{1}{n} \sum_{i= 1}^n \left< \hat{g}_k,  \Pcal_{\Mcal}(x_{i,k} - d_{i,1} - d_{i,2}  ) - [x_{i,k}  - d_{i,1}    ]\right>+\frac{1}{n}\sum_{i= 1}^n \left< \hat{g}_k,  d_{i,2}\right> \\
\leq & \frac{LQ }{n} \sum_{i=1}^n \|  d_i\|^2 +\frac{1}{n} \sum_{i= 1}^n \left< \hat{g}_k - \grad f_i(x_{i,k}),  d_{i,2}\right>\\
\leq & \frac{LQ }{n} \sum_{i=1}^n \|  d_i\|^2 + \frac{1}{4nL}\sum_{i= 1}^n
 \|\hat{g}_k - \grad f_i(x_{i,k})\|^2 + \frac{L}{n}\sum_{i= 1}^n \|d_{i,2}\|^2 \\
  \leq & \frac{LQ+L }{n} \|(I_{nd}-\bW^t)\bx_k + \alpha \bv_k\|^2 + \frac{L^2}{n}  \|\bar{\bx}_k - \bx_k\|^2 \\
   \leq & \frac{4LQ+3L  }{n}\|\bar{\bx}_k - \bx_k\|^2  + \frac{2LQ+2L }{n}\alpha^2 \|\bs_k\|^2.
\end{aligned}
\ee
Plugging \eqref{eq:b1-1} and \eqref{eq:b1-2} into \eqref{eq:b1-0} gives
\be \label{eq:b1}
\begin{aligned}
b_1 &\leq  \frac{4LQ+ 4 L}{n} \|\bx_k - \bar{\bx}_k \|^2 +\frac{2LQ+3L}{n} \alpha^2 \| \bs_k\|^2 -  \alpha  \|\hat{g}_k\|^2,
\end{aligned}
\ee
For $b_2$, applying Lemma \ref{lemma:quadratic} and Lemma \ref{lemma:dpgta:consensus} yields
\begin{equation} \label{eq:b2}
\begin{aligned}
      b_2 & \leq \frac{M_2^2}{\alpha^2 L n^2}(\| \bx_k - \bar{\bx}_k\|^4 + \| \bx_{k+1} - \bar{\bx}_{k+1}\|^4  ) \\
      &
     \leq \frac{M_2^2D L }{ n}(\| \bx_k - \bar{\bx}_k\|^2 + \| \bx_{k+1} - \bar{\bx}_{k+1}\|^2  ),
\end{aligned}
\end{equation}
Since $\|\bv_k\|^2 \leq \|\bs_k\|^2 \leq 4\zeta^2L^2$ in Lemma \ref{lemma:neibohood}. Since $\|\hat{v}_k\|^2 = \frac{1}{n^2} \|\sum_{i=1}^n v_{i,k}\|^2 \leq \frac{1}{n}\sum_i\|v_{i,k}\|^2 = \frac{1}{n}\|\bv_k\|^2$,  combining with \eqref{eq:sum-consen-grad} yields
\begin{equation} \label{eq:b3}
    \begin{aligned}
    b_3  
     \leq &\frac{3L}{4} \left(\frac{4(8Q + \sqrt{n}L_2 + M_2)^2}{n^2} \|\bx_k - \bar{\bx}_k\|^4 +\frac{16Q^2\alpha^4}{n^2} \| \bv_k \|^4  \right. \\
     & + \left. \frac{4\alpha^2}{n} \| \bv_k  \|^2  + \frac{4M_2^2}{n^2} \| \bx_{k+1}  - \bar{\bx}_{k+1} \|^4\right)\\
      \leq & \frac{4DL^3(8Q + \sqrt{n}L_2 + M_2)^2}{n}\alpha^2 \|\bx_k - \bar{\bx}_k\|^2   + \frac{4DL^3M_2^2}{n}\alpha^2 \| \bx_{k+1}  - \bar{\bx}_{k+1} \|^2 \\
     & + \frac{192Q^2L^3\alpha^4 + 3L\alpha^2}{n}\| \bs_k \|^2.
    \end{aligned}
\end{equation}
Then, combining \eqref{eq:b1}, \eqref{eq:b2}, and \eqref{eq:b3} gives
\begin{equation}\nonumber
\begin{aligned}
   &   f(\bar{x}_{k+1})  \leq  f(\bar{x}_k) + b_1 + \frac{\alpha^2 L^2}{2} \|\hat{g}_k\|^2 + b_2 + \frac{L}{n}\|\bx_k - \bar{\bx}_k\|^2 + b_3 \\
      \leq & f(\bar{\bx}_k) - ( \alpha - \frac{L^2}{2} \alpha^2 ) \|\hat{g}_k \|^2 + \frac{(2Q+ 3)L + 192 Q^2\alpha^2L^3 + 3L}{n} \alpha^2 \|  \bs_k\|^2\\
      & + \frac{4LQ+4L + M_2^2DL +  4DL^3(8Q + \sqrt{n}L_2 + M_2)^2\alpha^2}{n} \|\bx_k - \bar{\bx}_k\|^2  \\
      & + \frac{M_2^2D L + 4DL^3M_2^2\alpha^2}{ n}  \| \bx_{k+1}  - \bar{\bx}_{k+1} \|^2,
\end{aligned}
\end{equation}
where the second inequality utilizes $\frac{1}{n}\| \bx_k - \bar{\bx}_k \|^2 \leq  D L^2 \alpha^2$ in Lemma \ref{lemma:dpgta:consensus}.
\end{proof}

With these preparations, we give the proof of Theorem \ref{thm:dprgt}.

\begin{proof}\label{eq:proof1}[Proof of Theorem \ref{thm:dprgt}].
Since $\rho_t = 2\sigma_2^t \in (0,1)$, applying \cite[Lemma 2]{xu2015augmented} to \eqref{eq:dgta:consensus} yields 
\be\label{eq:sum:consensus}
\frac{1}{n}\sum_{k=0}^{K+1}\| \bx_k - \bar{\bx}_k \|^2 \leq \tilde{C}_0 \frac{4\alpha^2}{n}\sum_{k=0}^{K+1}\| \bs_k \|^2 + \tilde{C}_1,
\ee 
where $\tilde{C}_0 = \frac{2}{(1-\rho_t)^2}$ and $\tilde{C}_1 = \frac{2}{1-\rho_t^2}\frac{1}{n}\|\bx_0 - \bar{\bx}_0\|^2$. It follows from Lemma \ref{lemma:dgta:majorized} that
\begin{equation} \label{eq:desc}
\begin{aligned}
   &  f(\bar{x}_{k+1}) \\
      \leq & f(\bar{x}_0) - (\alpha  - \frac{\alpha^2L^2}{2})\sum_{k=0}^{K} \|\hat{g}_k\|^2 + \mathcal{C}_1 \alpha^2 \frac{1}{n} \sum_{k=0}^{K} \|\bs_k \|^2 +(\mathcal{C}_2 + \mathcal{C}_3) \frac{1}{n}\sum_{k=0}^{K+1} \| \bx_k - \bar{\bx}_k\|^2\\
      \leq & f(\bar{x}_0) - \frac{\alpha    }{2} \sum_{k=0}^{K} \|\hat{g}_k\|^2 + (\mathcal{C}_1 + 4\tilde{C}_0(\mathcal{C}_2 + \mathcal{C}_2)) \alpha^2 \frac{1}{n} \sum_{k=0}^{K+1} \|\bs_k \|^2 +(\mathcal{C}_2 + \mathcal{C}_3)\tilde{C}_1 \\
       \leq & f(\bar{x}_0) - \frac{\alpha    }{2} \sum_{k=0}^{K} \|\hat{g}_k\|^2 + (\mathcal{C}_1 + 4\tilde{C}_0(\mathcal{C}_2 + \mathcal{C}_2)) \alpha^2 \frac{1}{n} \sum_{k=0}^{K} \|\bs_k \|^2 \\
       & + 4(\mathcal{C}_1 + 16\tilde{C}_0(\mathcal{C}_2 + \mathcal{C}_2)) \alpha^2L^2 +(\mathcal{C}_2 + \mathcal{C}_3)\tilde{C}_1 \\
        \leq & f(\bar{x}_0) - \frac{\alpha    }{2} \sum_{k=0}^{K} \|\hat{g}_k\|^2 + (\mathcal{C}_1 + 4\tilde{C}_0(\mathcal{C}_2 + \mathcal{C}_2)) \alpha^2 \frac{1}{n} \sum_{k=0}^{K} \|\bs_k \|^2 + \tilde{C}_4,
     \end{aligned}
\end{equation}
where we use $\alpha< \frac{1}{L^2}, \frac{1}{n}\|\bs_k\|^2 \leq 16 L^2$ and $\tilde{C}_4: = 4(\mathcal{C}_1 + 16\tilde{C}_0(\mathcal{C}_2 + \mathcal{C}_2)) \alpha^2  L^2 +(\mathcal{C}_2 + \mathcal{C}_3)\tilde{C}_1$. Next, we need to associate $\|\hat{g}_k\|^2$ with $\|\bs_k\|^2$. Since  $\| \bs_k  \|  \leq \|\bs_k - \hat{\bf G}_k \| + \|\hat{\bf G}_k \|$, we have 
\be\label{eq:tmp10}
-\sum_{k=0}^{K} \|\hat{g}_k \|^2  = -\frac{1}{n}\sum_{k=0}^{K} \|\hat{\bf G}_k \|^2  \leq \frac{1}{n} \sum_{k=0}^{K}\|\bs_k - \hat{\bf G}_k \|^2 - \frac{1}{2n}\sum_{k=0}^{K}\|\bs_k\|^2.
\ee
Applying \cite[Lemma 2]{xu2015augmented} into \eqref{eq:diff-s-g} yields
\bee
\begin{aligned}
\frac{1}{n}\sum_{k=0}^{K}\|\bs_k - \hat{\bf G}_k \|^2 &\leq \tilde{C}_2\frac{1}{n}\sum_{k=0}^{K}\|{\bf G}^{k+1} - {\bf G}^k \|^2 + \tilde{C}_3 \\
& \overset{\eqref{eq:diff-g-g}}{ \leq} \tilde{C}_2\frac{1}{n}\sum_{k=0}^{K} (32L^2\|\bx_k - \bar{\bx}_k \|^2 + 8L^2\alpha^2 \|\bs_k\|^2  ) + \tilde{C}_3 \\
& \overset{\eqref{eq:sum:consensus}}{ \leq} \tilde{C}_2(128\tilde{C}_0  + 8)L^2\alpha^2\frac{1}{n}\sum_{k=0}^{K}  \|\bs_k\|^2   + 32\tilde{C}_1\tilde{C}_2L^2  + \tilde{C}_3,
\end{aligned}
\eee
where $\tilde{C}_2 = \frac{2}{(1-\sigma_2^t)^2}$ and $\tilde{C}_3 = \frac{2}{1-\sigma^2t_2}\frac{1}{n}\|\bs_0 - \hat{\mathbf{G}}_0\|^2$.
Plugging the above inequality into \eqref{eq:tmp10} leads to
\bee
-\sum_{k=0}^{K} \|\hat{g}_k \|^2    \leq  \left[ \tilde{C}_2(128\tilde{C}_0  + 8)L^2\alpha^2 - \frac{1}{2}  \right]  \frac{1}{n}\sum_{k=0}^{K}\|\bs_k\|^2  + 32\tilde{C}_1\tilde{C}_2L^2  + \tilde{C}_3.
\eee
Since $\alpha <\min\left\{1, \frac{1}{4\left(\tilde{C}_2(128\tilde{C}_0  + 8)L^2 + 2(\mathcal{C}_1 + 4\tilde{C}_0(\mathcal{C}_2 + \mathcal{C}_2)) \right) }\right\} $,
it follows from \eqref{eq:desc} that
\bee
\begin{aligned}
&   f(\bar{x}_{k+1})\\
\leq &  f(\bar{x}_0) - \frac{\alpha    }{2}   \left[\frac{1}{2} -  \tilde{C}_2(128\tilde{C}_0  + 8)L^2\alpha^2 - 2(\mathcal{C}_1 + 4\tilde{C}_0(\mathcal{C}_2 + \mathcal{C}_2)) \alpha  \right] \frac{1}{n}  \sum_{k=0}^{K}\|\bs_k\|^2 \\
      &  +(\tilde{C}_4 + 16\tilde{C}_1\tilde{C}_2L^2 \alpha + \frac{\alpha}{2}\tilde{C}_3 ) \\
       \leq & f(\bar{x}_0) - \frac{\alpha    }{8}\frac{1}{n}     \sum_{k=0}^{K}\|\bs_k\|^2   +(\tilde{C}_4 + 16\tilde{C}_1\tilde{C}_2L^2 \alpha + \frac{\alpha}{2}\tilde{C}_3 ). 
     \end{aligned}
\eee
This implies
\bee
\frac{\alpha}{8} \sum_{k=0}^{K} \|\hat{g}_k \|^2  \leq \frac{\alpha}{8} \frac{1}{n} \sum_{k=0}^{K} \|\bs_k\|^2 \leq f(\bar{x}_0) - f^* + \tilde{C}_5,
\eee
where $\tilde{C}_5 := \tilde{C}_4 + 16\tilde{C}_1\tilde{C}_2L^2 \alpha + \frac{\alpha}{2}\tilde{C}_3$ and $f^* = \min_{x\in \Mcal} f(x)$. Furthermore, 
\bee
\min_{k=0,\cdots,K} \| \hat{g}_k\|^2 = \min_{k=0,\cdots,K} \| \hat{s}_k\|^2 \leq  \min_{k=0,\cdots,K} \frac{1}{n}\|\bs_k\|^2 \leq \frac{8(f(\bar{x}_0) - f^* + \tilde{C}_5)}{\alpha K}.
\eee
Then, it follows from \eqref{eq:sum:consensus} that
\bee
\min_{k=0,\cdots,K} \frac{1}{n} \|\bx_k - \bar{\bx}_k \|^2  \leq \frac{32\tilde{C}_0\alpha(f(\bar{\bx}_0) - f^* + \tilde{C}_5) + \tilde{C}_1}{ K}.
\eee
Since 
$
\|\grad f(\bar{x}_k) \|^2  \leq 2\| \hat{g}_k\|^2 + 2\|\grad f(\bar{x}_k)  - \hat{g}_k \|^2  \leq 2\| \hat{g}_k\|^2 + \frac{2L^2}{n}\| \bx_k - \bar{\bx}_k\|^2. 
$
We conclude that 
\bee
\min_{k=0,\cdots,K} \|\grad f(\bar{x}_k) \|^2  \leq \frac{(16 +64\tilde{C}_0\alpha^2 )(f(\bar{x}_0) - f^* + \tilde{C}_5)+ 2L^2 \tilde{C}_1\alpha}{\alpha K}.
\eee
The proof is completed.
\end{proof}

\section{Numerical experiments}
In this section, we compare our proposed DPRGD (Algorithm \ref{alg:drpgd}) and DPRGT (Algorithm \ref{alg:drgta}) with DRDGD and DRGTA in \cite{chen2021decentralized} on three problems: Decentralized principal component analysis, decentralized generalized eigenvalue problem, and decentralized low-rank matrix completion.

\subsection{Decentralized principal component analysis} \label{sub:pca}
We perform numerical tests on the following decentralized principal component analysis problem:
\be \label{prob:pca}
\min _{\mathbf{x} \in \mathcal{M}^{n}}-\frac{1}{2 n} \sum_{i=1}^{n} \text{tr}(x_{i}^{\top} A_{i}^{\top} A_{i} x_{i}), \quad \text { s.t. } \quad x_{1}=\ldots=x_{n},
\ee
where $\Mcal^n:=\underbrace{{\rm St}(d,r) \times \cdots \times {\rm St}(d,r)}_{n}$, $A_{i} \in \mathbb{R}^{m_{i} \times d}$ is the local data matrix in $i$-th agent, and $m_{i}$ is the sample size.
Note that for any solution $x^*$ of \eqref{prob:pca}, $x^*Q$ with an orthogonal matrix $Q \in \R^{r\times r}$ is also a solution. We use the function
\[ d_s(x, x^*) := \min_{Q\in \R^{r\times r},\; Q^\top Q = QQ^\top = I_d} \; \|xQ - x^*\| \]
to compute the distance between two points $x$ and $x^*$.

\subsubsection{Synthetic dataset}
We fix $m_{1}=\ldots=m_{n}=1000, d=10$, and $r=5$. We then generate a matrix $B \in \R^{1000n  \times d}$ and do the singular value decomposition
\[ B = U \Sigma V^\top, \]
where $U \in \R^{1000n \times d}$ and $V \in \R^{d\times d}$ are orthogonal matrices, and $\Sigma \in \R^{d \times d}$ is a diagonal matrix. To control the distributions of the singular values, we set $\tilde{\Sigma} = {\rm diag} (\xi^j)$ with $\xi \in (0,1)$. Then, $A$ is set as
\[ A = U \tilde{\Sigma} V^\top \in \R^{1000n \times d}. \]
$A_i$ is obtained by randomly splitting the rows of $A$ to $n$ subsets with equal cardinalities. It is easy to check the first $r$ columns of $V$ form the solution of \eqref{prob:pca}. In the experiments, we set $\xi$ and $n$ to $0.8$ and $8$, respectively.

We employ fixed step sizes for all algorithms. For DPRGT and DRGTA, we use the step size $\alpha=\frac{\hat{\beta}n}{\sum_{i=1}^n m_i}$. For DPRGD and DRDGD, the step size is set to $\alpha = \frac{\hat{\beta}}{\sqrt{K}}$ with $K$ being the maximal number of iterations. The grid search is utilized to find the best $\hat{\beta}$ for each algorithm. We choose the polar decomposition as the retraction operators for DRDGD and DRGTA. We test several graph matrices to model the topology across the agents, namely, the Erdos-Renyi (ER) network with probability $p = 0.3, 0.6$, and the Ring network. Throughout this section, we select the mixing matrix $W$ to be the Metropolis constant edge weight matrix \cite{shi2015extra}.

The results of different algorithms are presented in Figures \ref{fig:num-pca-consensus}, \ref{fig:num-pca-graph}, and \ref{fig:num-pca-alg}. It can be seen from Figure \ref{fig:num-pca-consensus} that the single-step consensus (i.e., $t=1$) gives very close performance to those of multiple-step consensuses (i.e., $t=10$) although the convergence restricts a large $t$. This phenomenon has also been observed in \cite{chen2021decentralized}.  In Figure \ref{fig:num-pca-graph}, DPRGT performs very similarly under different graphs. For DPRGD, a densely connected graph (e.g., ER $p=0.6$) will help to obtain a better final solution. For Figure \ref{fig:num-pca-alg}, we see DPRGT and DRGTA have very close trajectories on the consensus error, the objective function, the gradient norm, and the distance to the global optimum. This also occurs for DPRGD and DRDGD. In fact, it follows from Lemma \ref{lemma-project} that the difference between the updates in DPRGD and DRDGD is controlled by the square of the consensus error. As shown in Figure \ref{fig:num-pca-alg}, DPRGD and DRDGD do not converge to stationary points under the choice of fixed step sizes.

\begin{figure}[htp]
	\centering
	\includegraphics[width = 0.45 \textwidth]{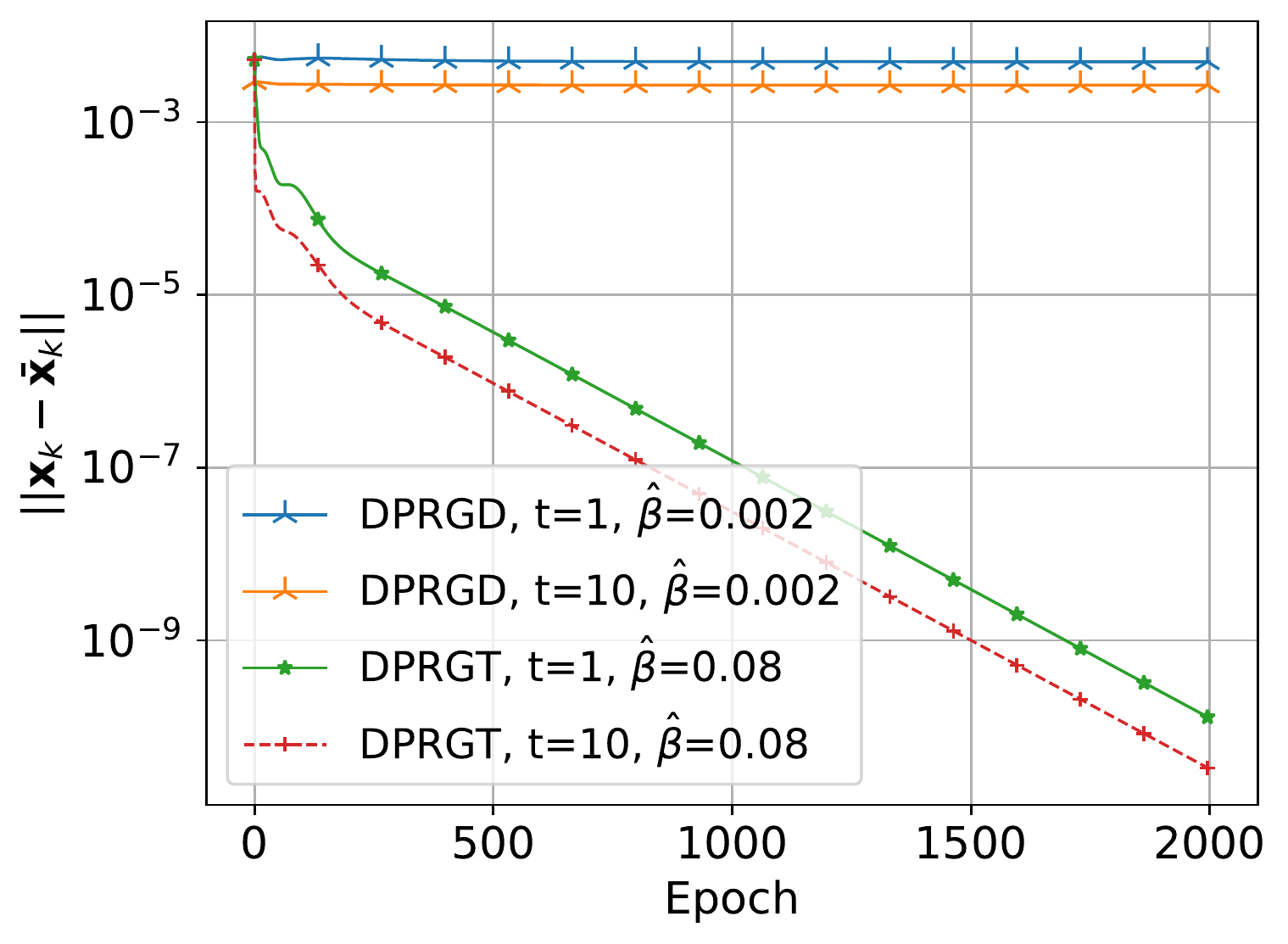}
	\includegraphics[width = 0.45 \textwidth]{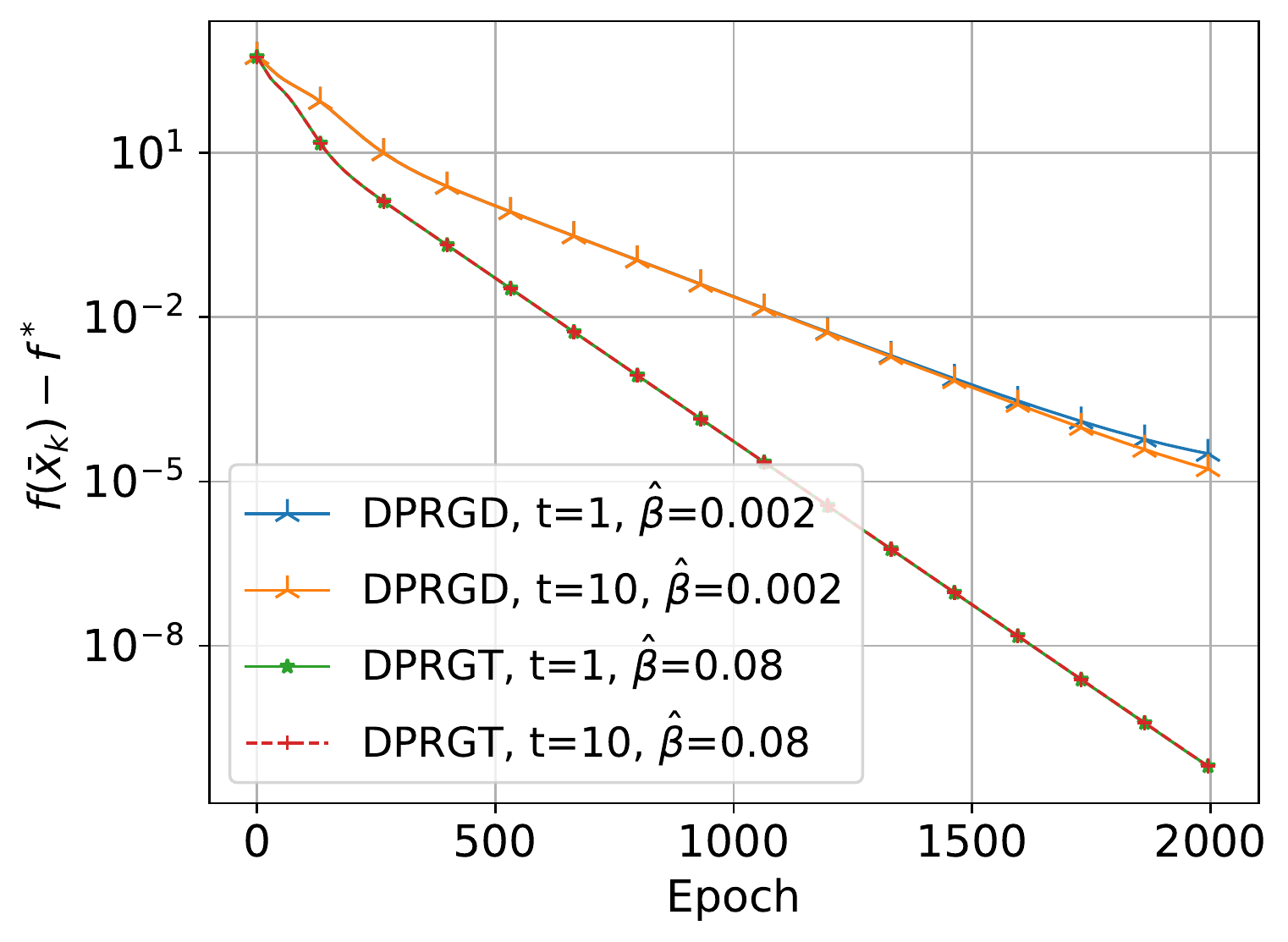} \\
	\includegraphics[width = 0.45 \textwidth]{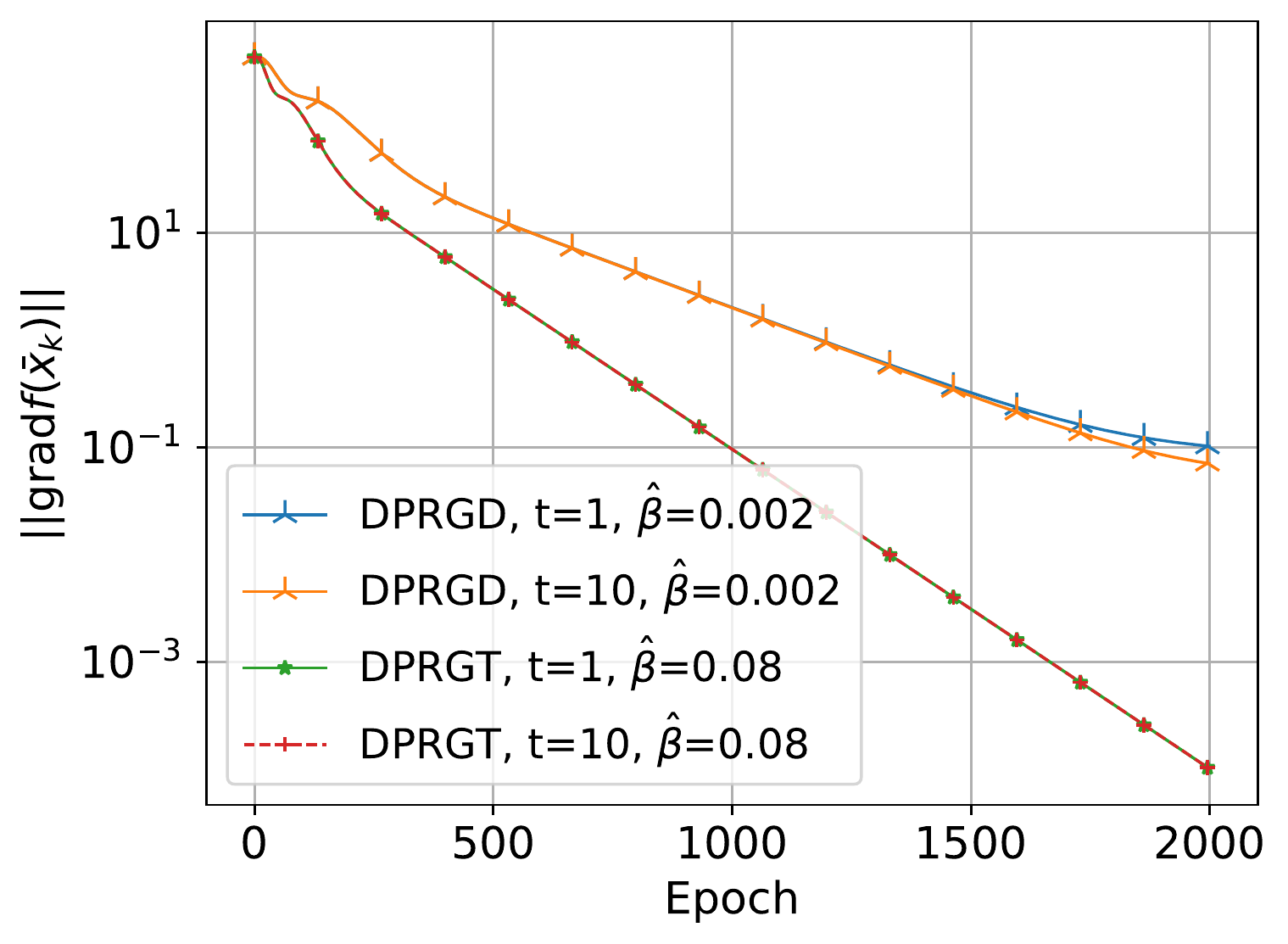}
	\includegraphics[width = 0.45 \textwidth]{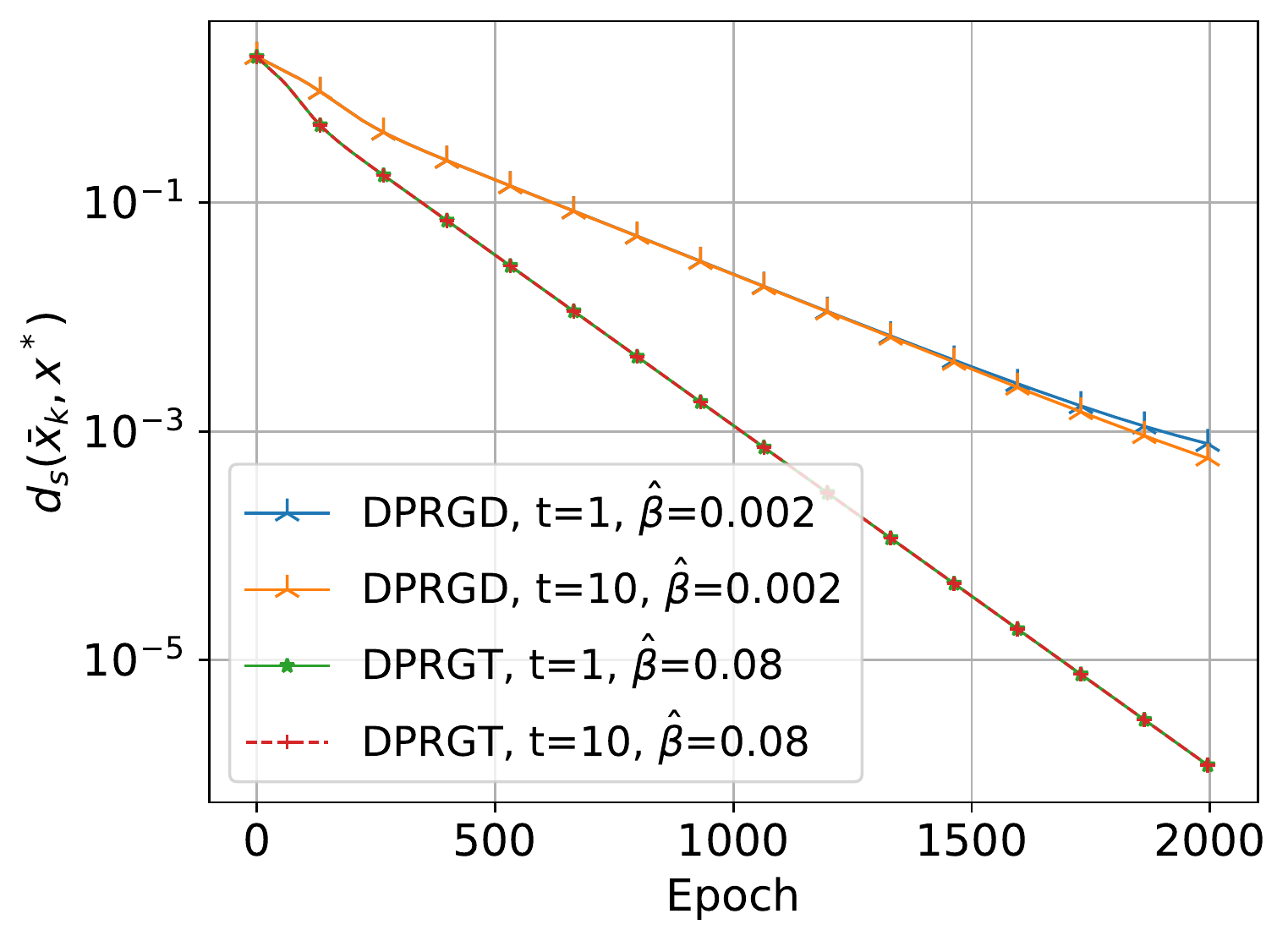}
	\caption{Numerical results on synthetic dataset with different numbers of consensus steps on graph ER $p=0.6$.}	
	\label{fig:num-pca-consensus}
\end{figure}

\begin{figure}[htp]
	\centering
	\includegraphics[width = 0.45 \textwidth]{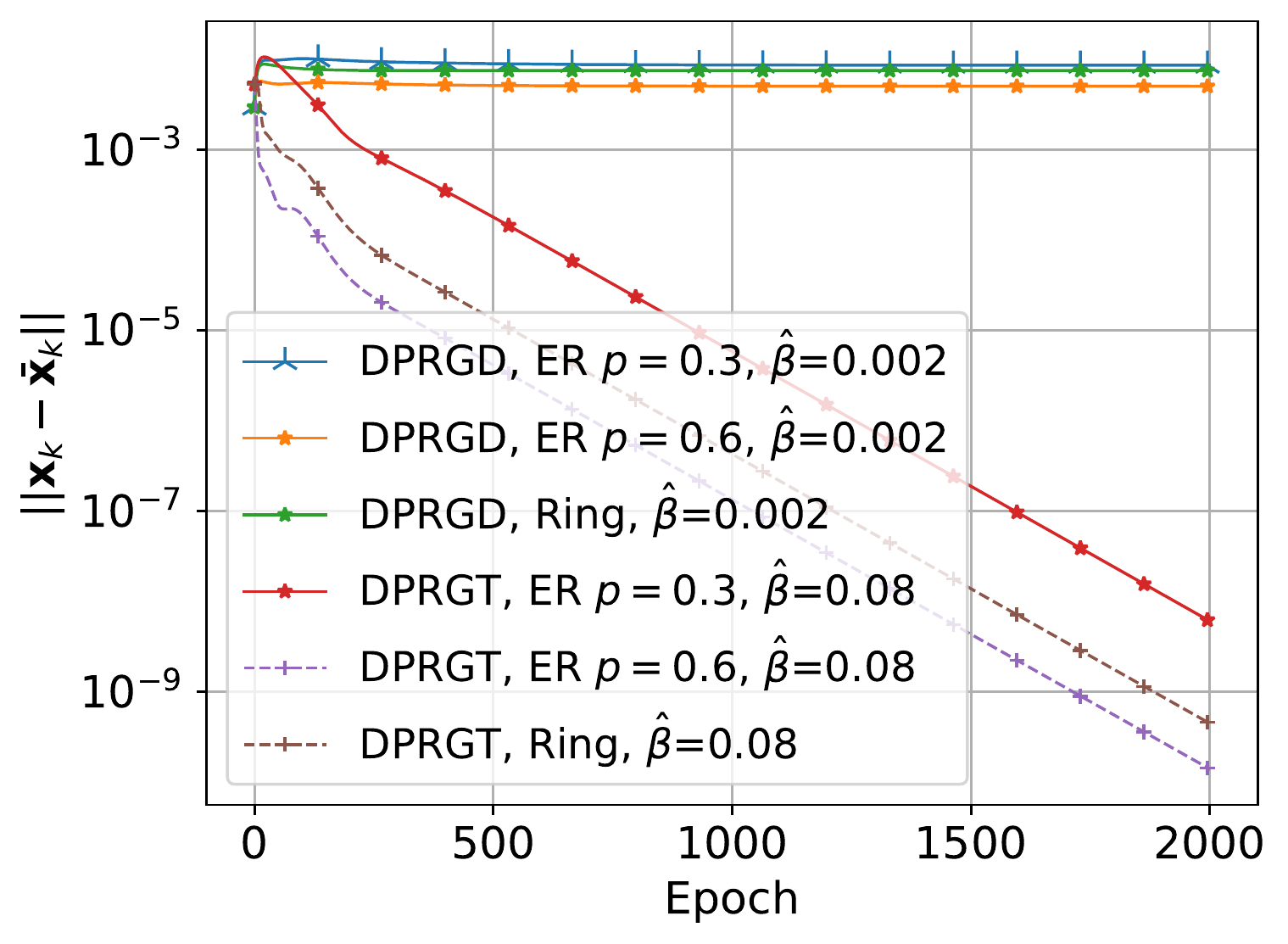}
	\includegraphics[width = 0.45 \textwidth]{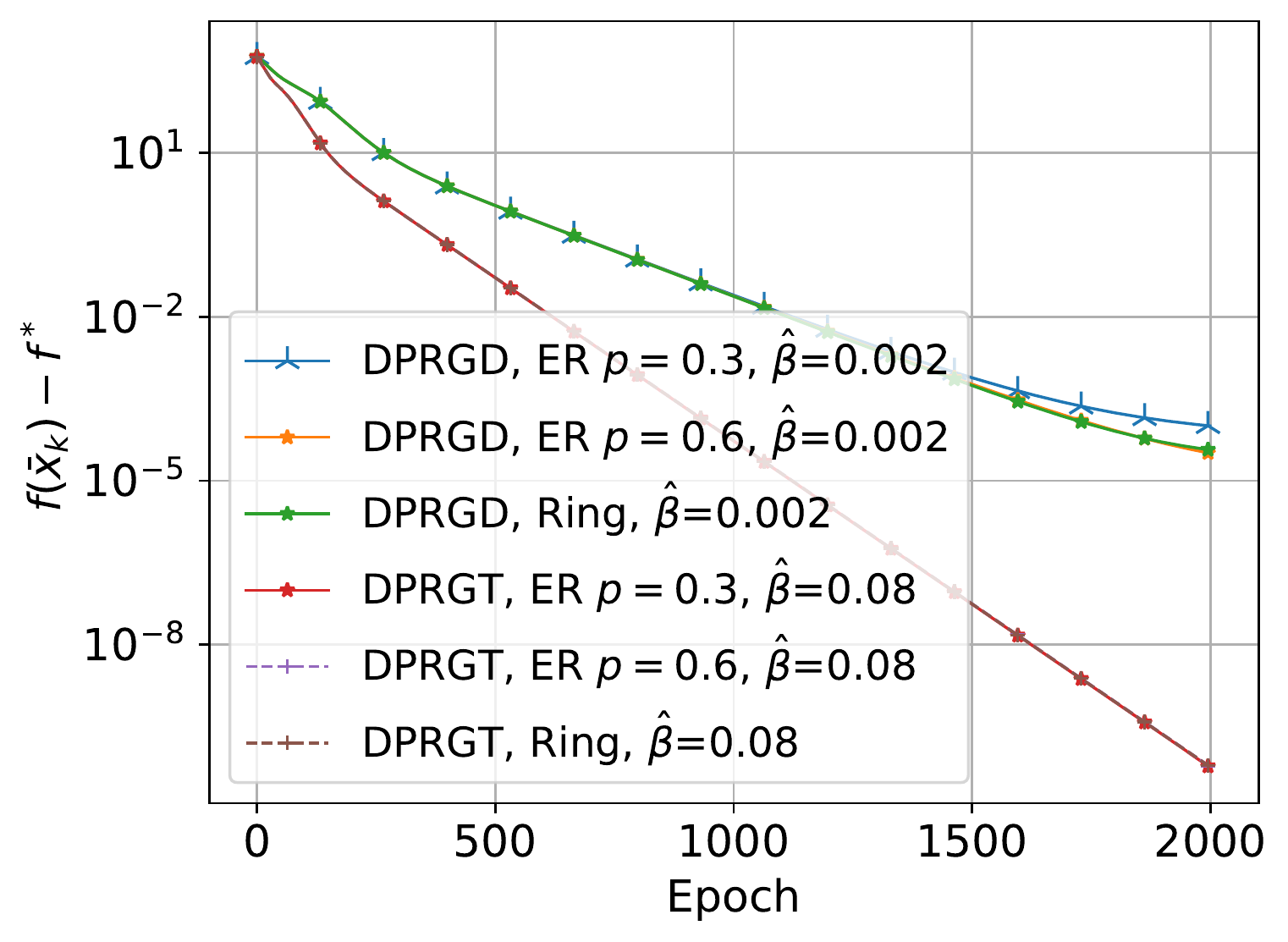} \\
	\includegraphics[width = 0.45 \textwidth]{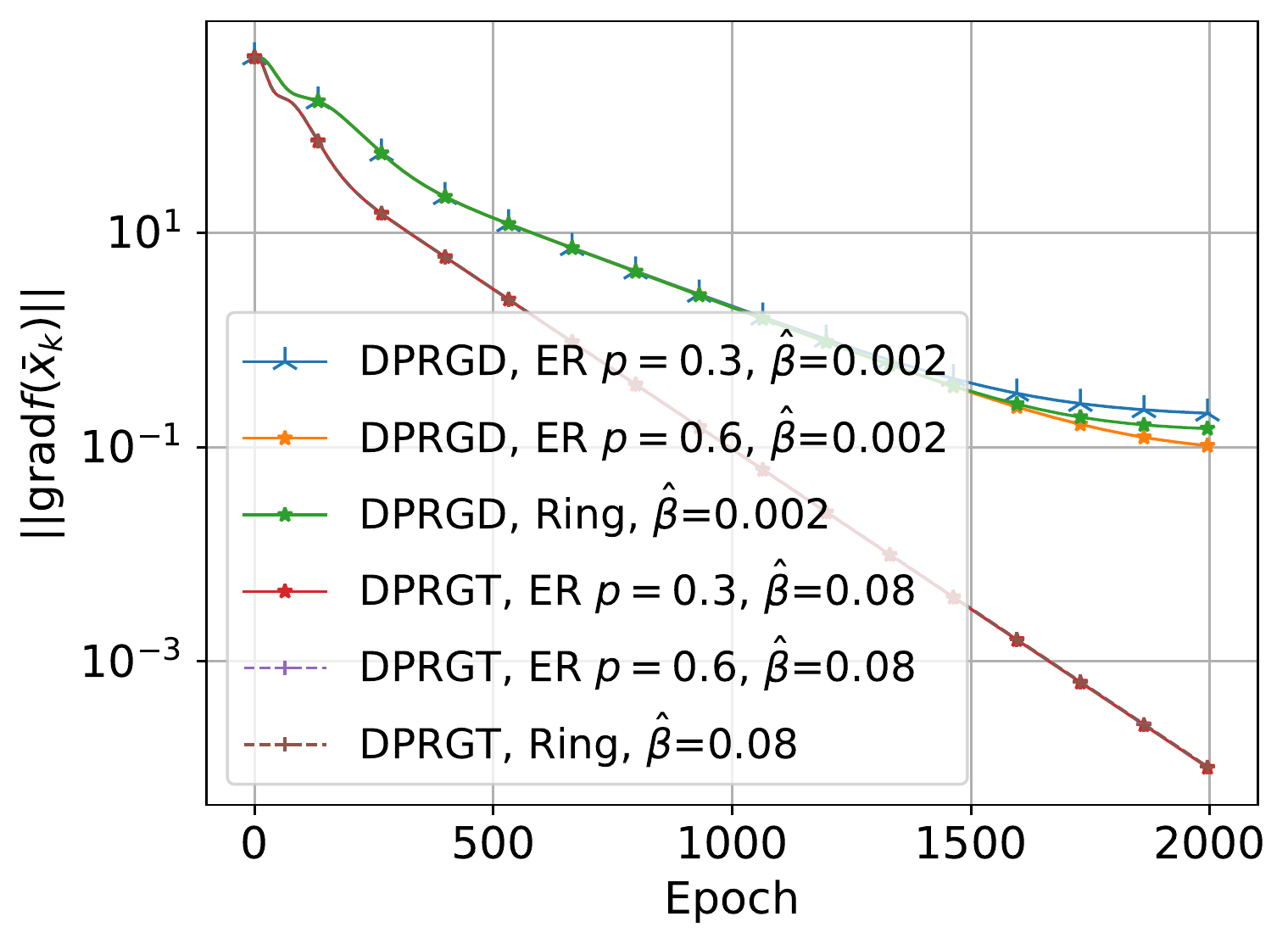}
	\includegraphics[width = 0.45 \textwidth]{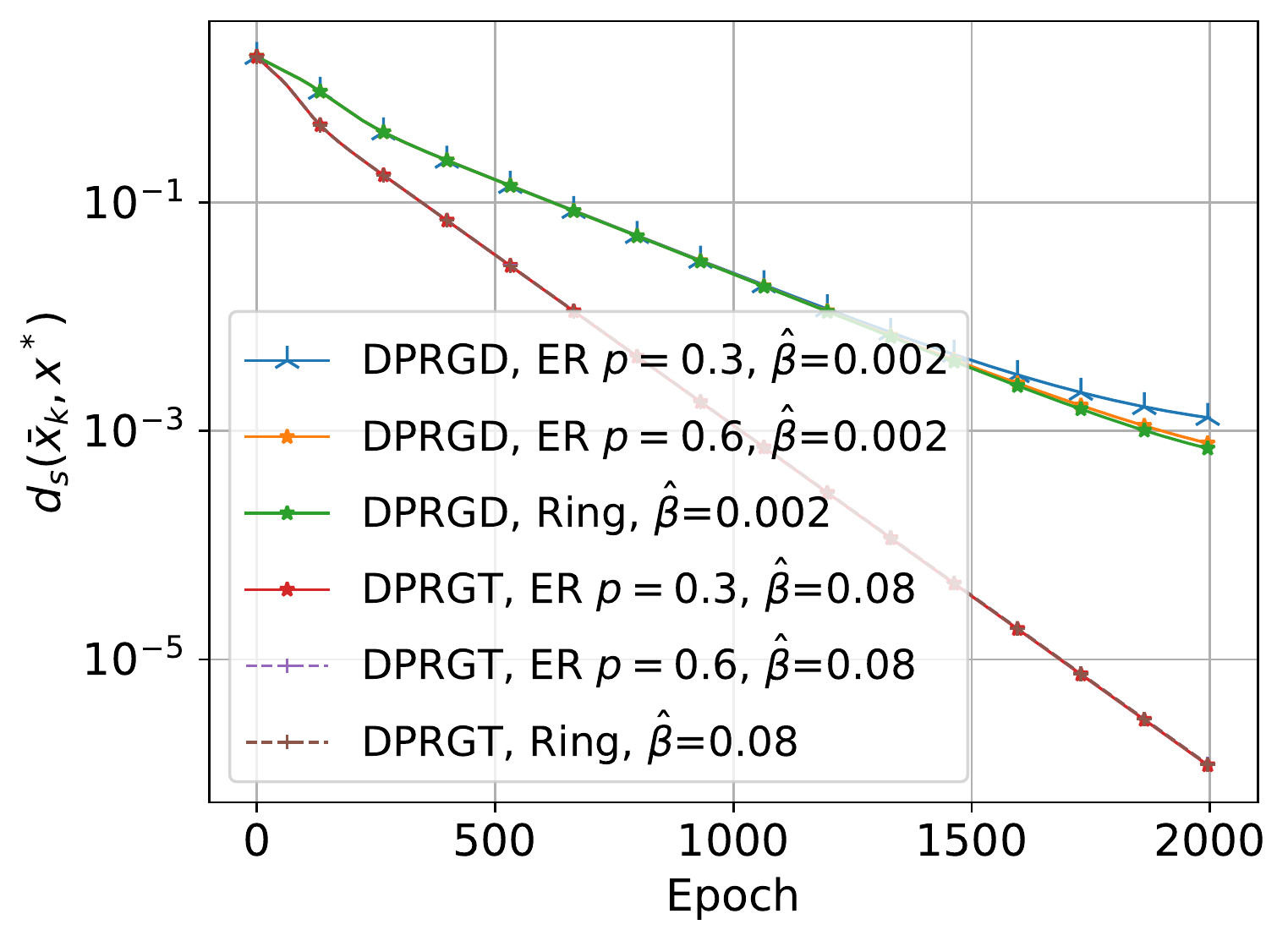}
	\caption{Numerical results on the synthetic dataset with different network graphs and single-step consensus.}	
	\label{fig:num-pca-graph}
\end{figure}

\begin{figure}[htp]
	\centering
    \includegraphics[width = 0.45 \textwidth]{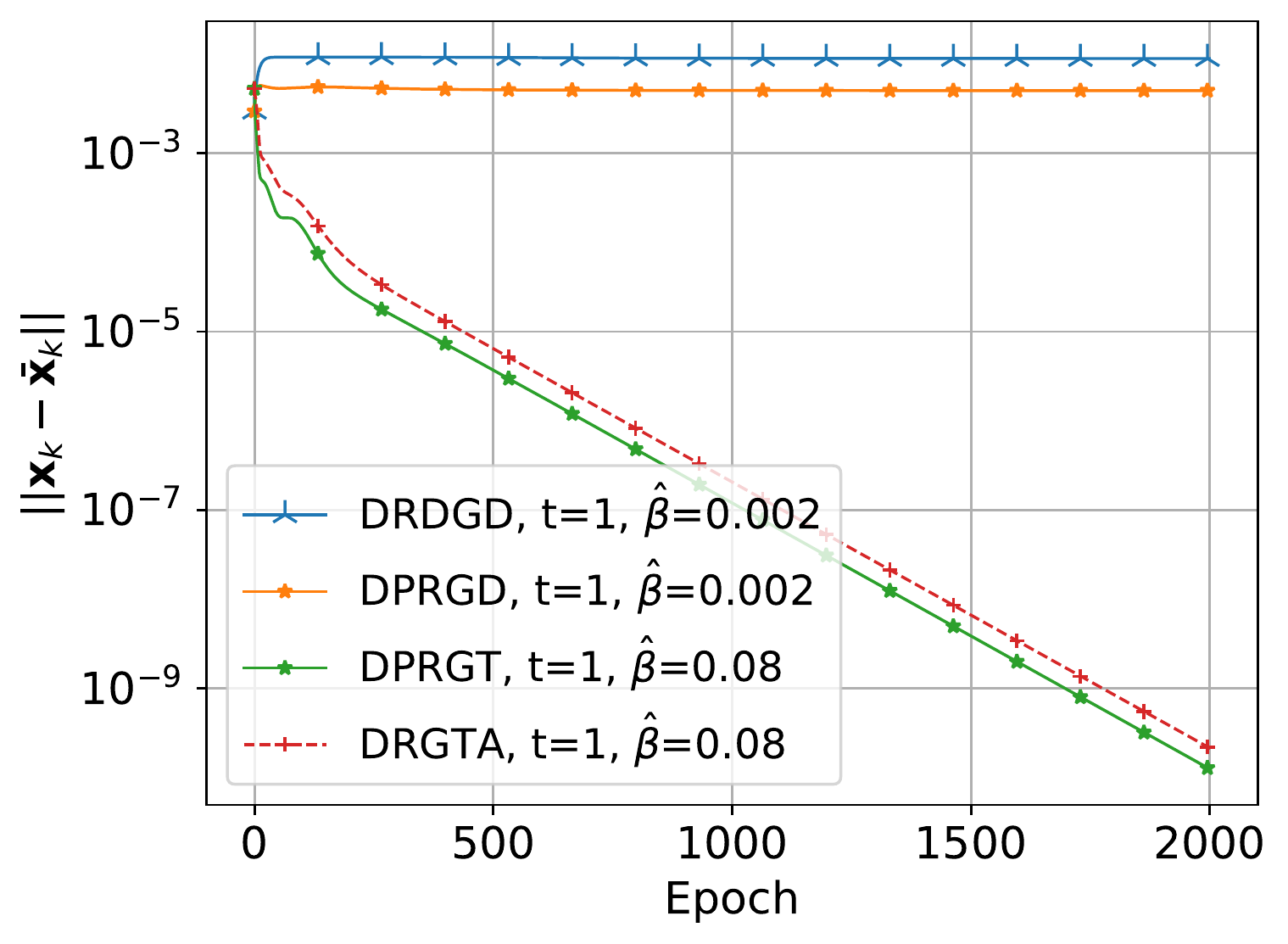}
    \includegraphics[width = 0.45 \textwidth]{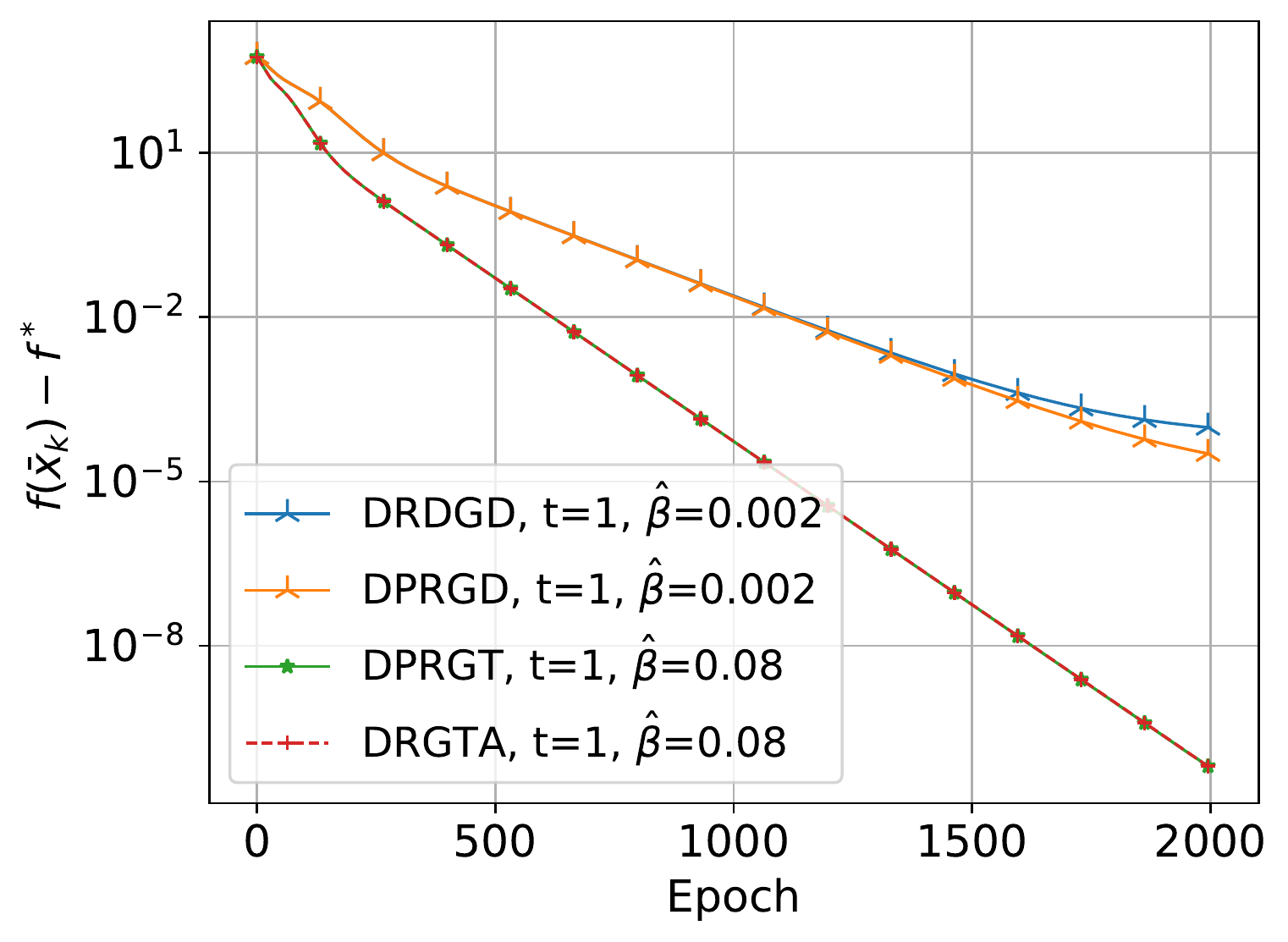} \\
    \includegraphics[width = 0.45 \textwidth]{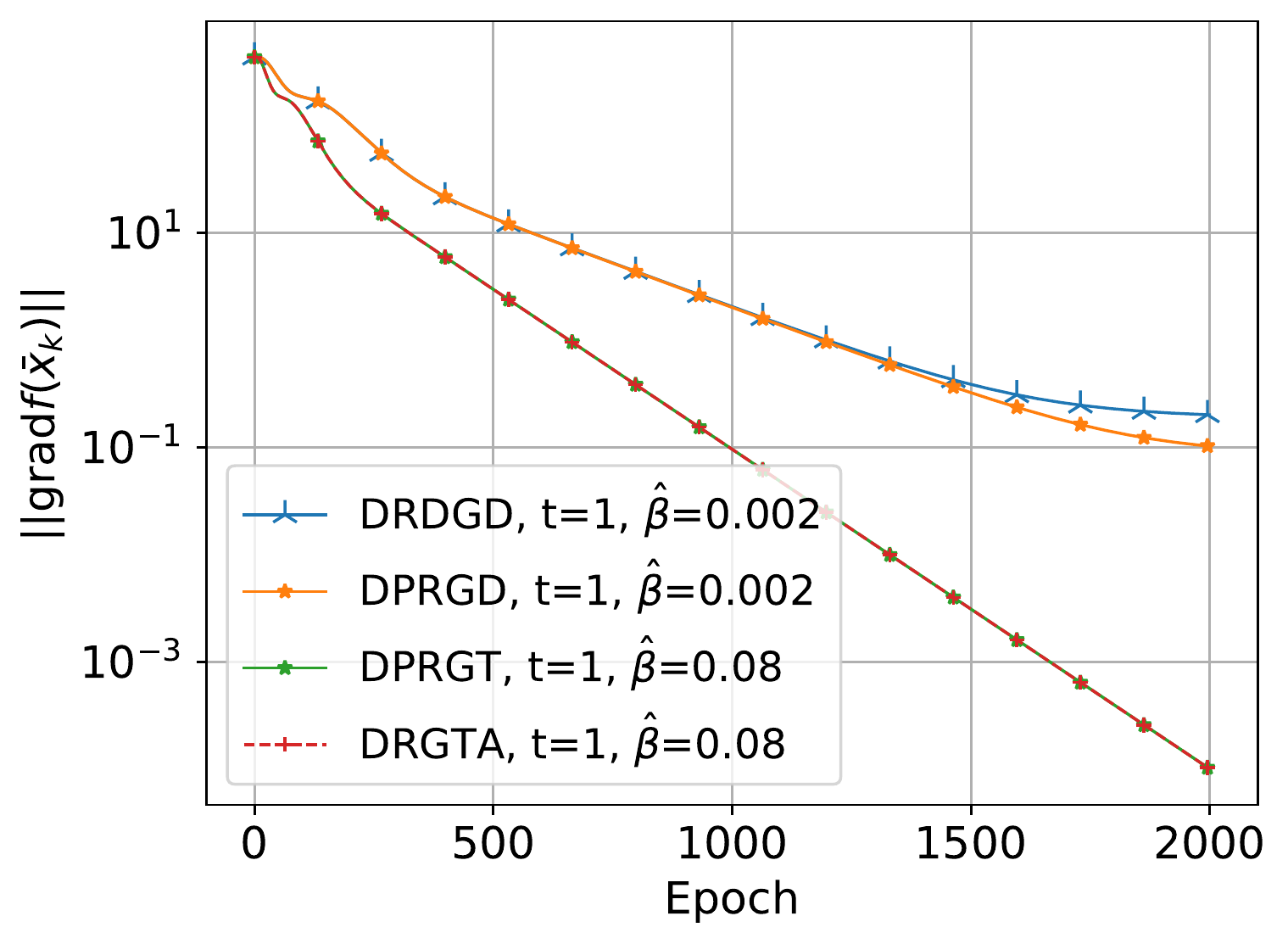}
    \includegraphics[width = 0.45 \textwidth]{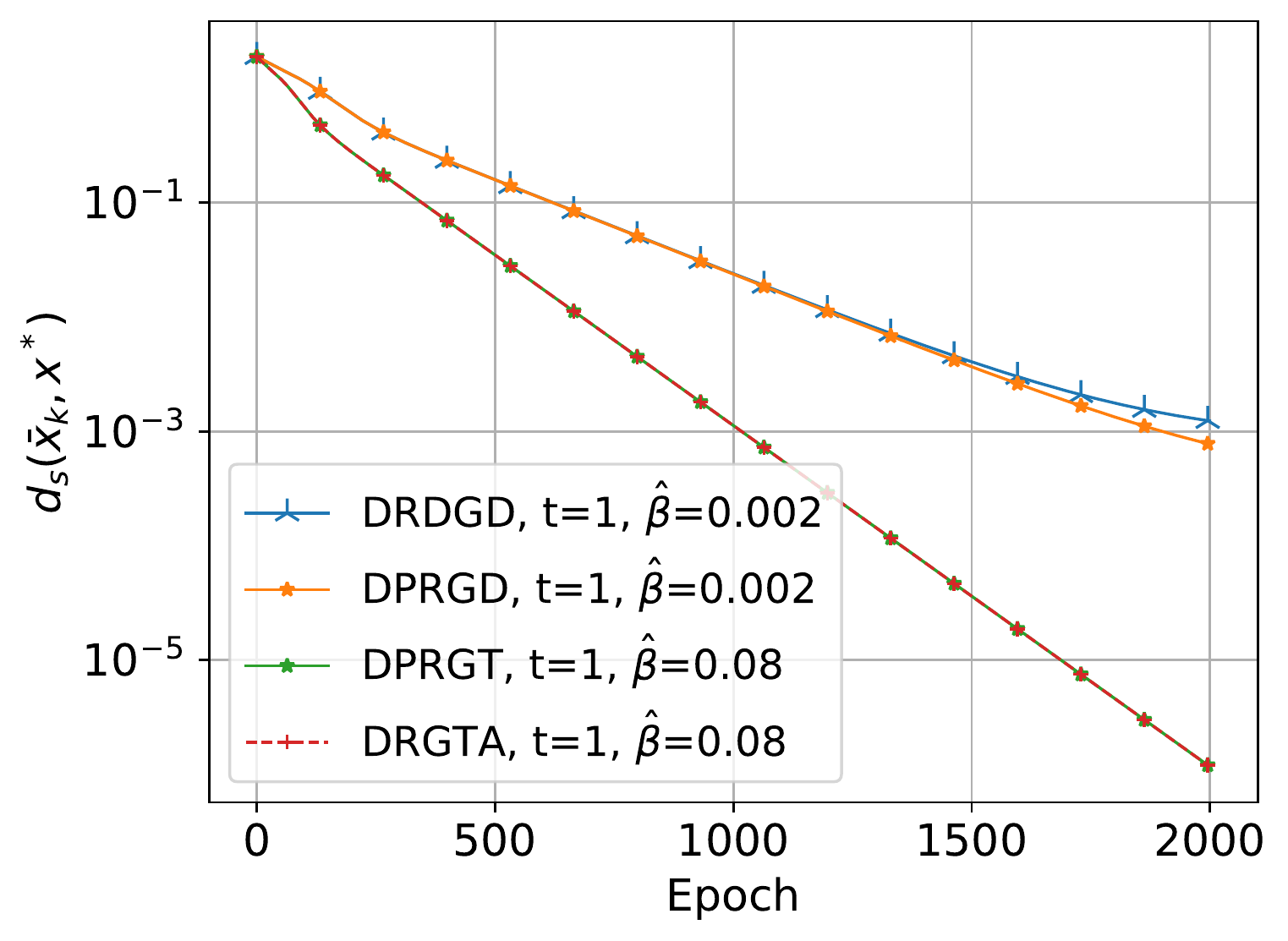}
	\caption{Numerical results on the synthetic dataset with ER $p=0.6$.}	
	\label{fig:num-pca-alg}
\end{figure}



\subsubsection{Mnist dataset}
To evaluate the efficiency of our proposed method, we also do numerical tests on the Mnist dataset \cite{lecun1998mnist}. The testing images consist of 60000 handwritten images of size $32 \times 32$ and are used to generate $A_i$'s. We first normalize the data matrix by dividing 255 and randomly split the data into $n=8$ agents with equal cardinality. Then, each agent holds a local matrix $A_i$ of dimension $\frac{60000}{n} \times 784$. We compute the first 5 principal components, i.e., $d=784, r=5$.

For all algorithms, we use the fixed step sizes $\alpha = \frac{\hat{\beta}}{60000}$ with a best-chosen $\hat{\beta}$.
Similar to the above setting, we see from Figure \ref{fig:num-pca-real} that DPRGT is very close to DRGTA, and both converge much faster than DPRGD and DRDGD.

\begin{figure}[htp]
	\centering
	\includegraphics[width = 0.45 \textwidth]{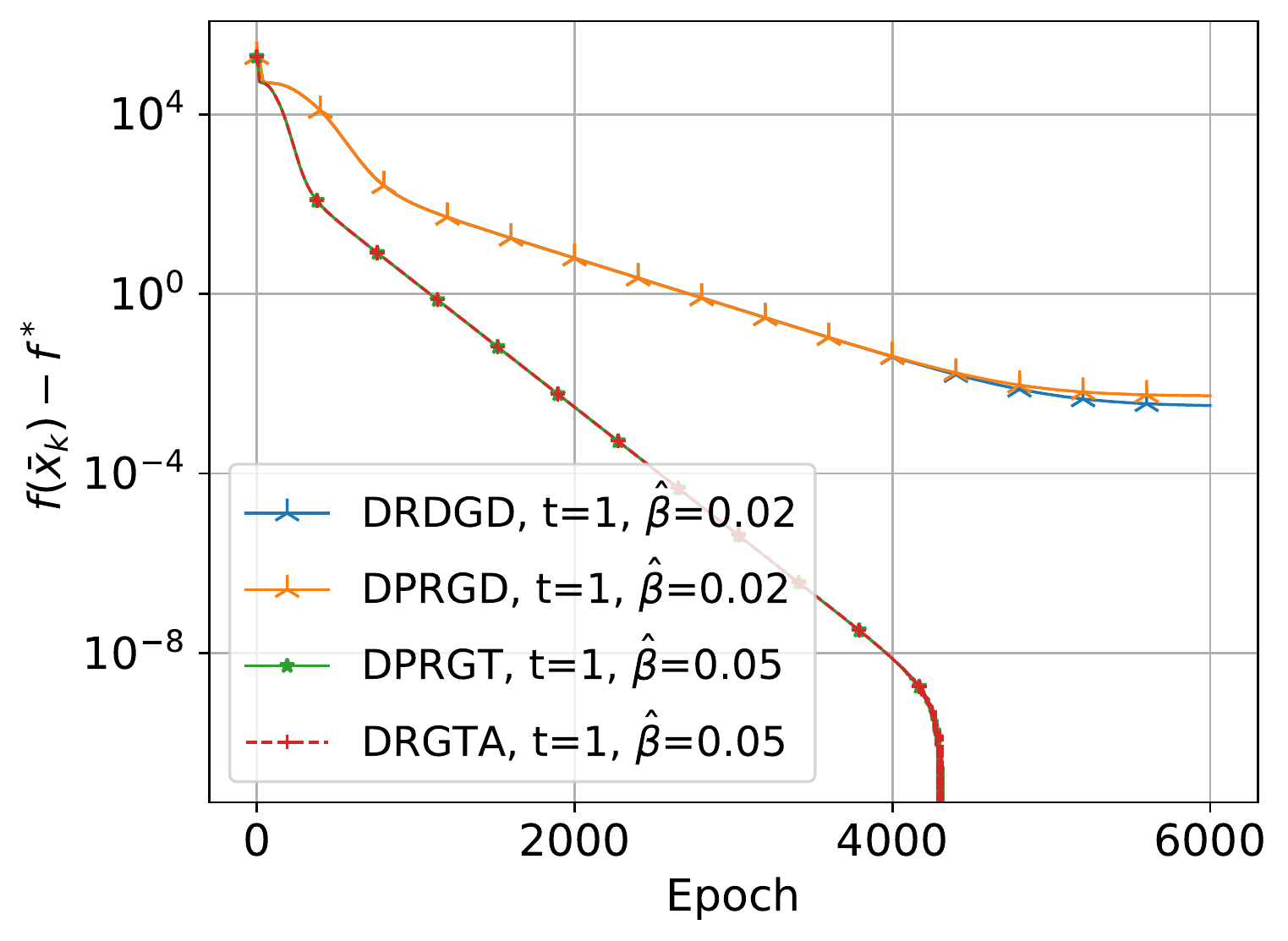}
	\includegraphics[width = 0.45 \textwidth]{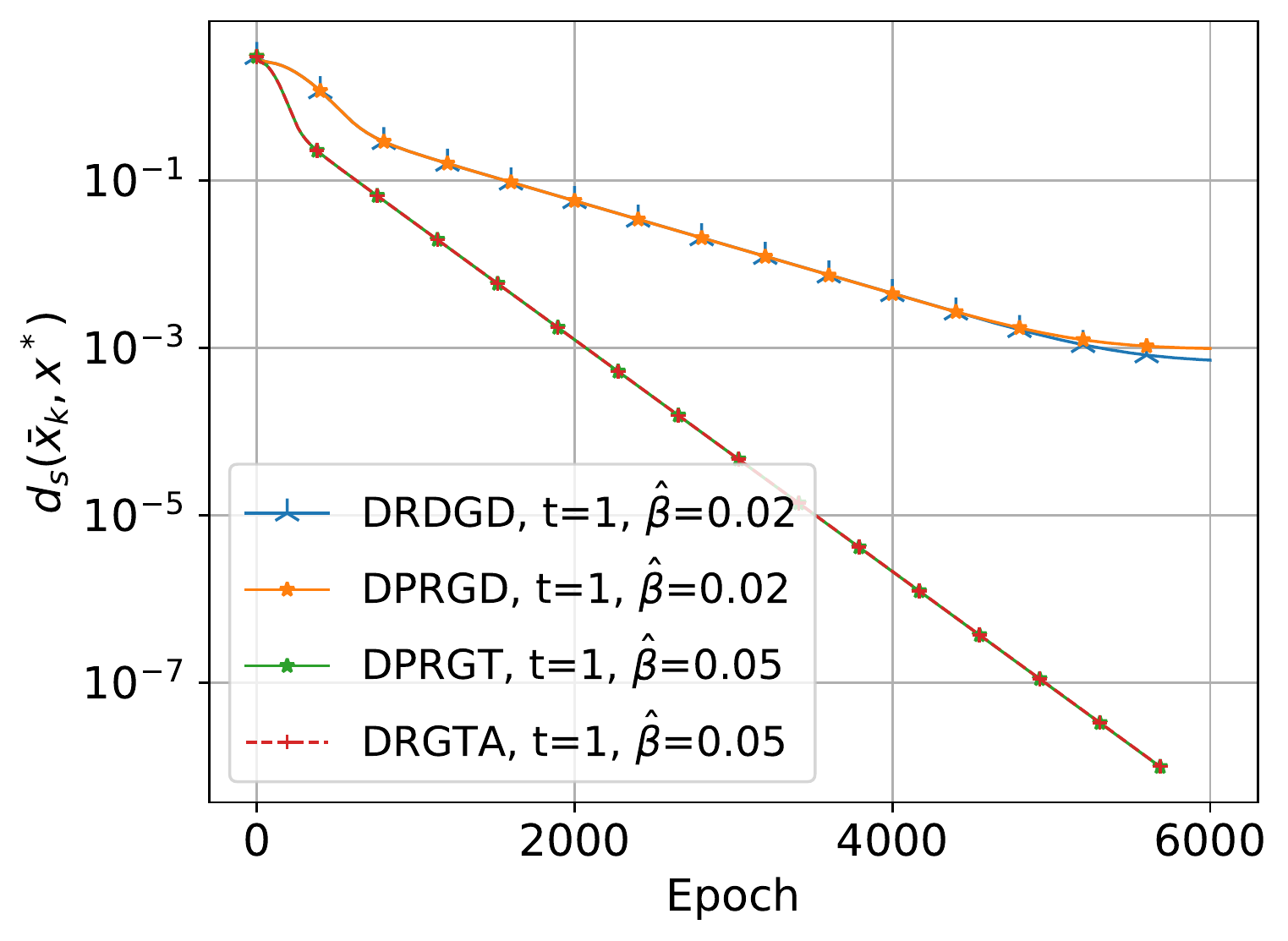}
	\caption{Numerical results on Mnist dataset with graph ER $p=0.3$.}	
	\label{fig:num-pca-real}
\end{figure}



\subsection{Decentralized generalized eigenvalue problem}
Consider the following decentralized generalized eigenvalue problem:
\be \label{prob:geig}
\begin{aligned}
\min_{\bx} \quad &\frac{1}{2n} \sum_{i=1}^{n} x_i^\top A_i^\top A_i x_i,  \\
\st \quad & x_1=\cdots = x_n, \\
& x_i^\top B x_i = I_r, \quad \forall i \in [n],
\end{aligned} \ee
where $B$ is a positive definite matrix. The optimization algorithms for solving the centralized version of \eqref{prob:geig} have been well studied in \cite{golub2013matrix}. Applications with form \eqref{prob:geig} appear in the kernel supervised PCA \cite{barshan2011supervised}, Fisher discriminant analysis \cite{fisher1936use}, and canonical correlation analysis \cite{hotelling1992relations,chaudhuri2009multi}.

In the test, we set $n=8$ and use the same way in Section \ref{sub:pca} to generate $A_i$'s. For $B$, we first generate an orthogonal matrix $Q \in \R^{d \times d}$ and then set $B = Q \Lambda Q^\top$ with $\Lambda = {\rm diag}(1.1, 1.1^{0.5}, \ldots, 1.1^{d/2-0.5})$. We set $d=10$ and $r=5$, i.e., $x_i \in \R^{10\times 5}$ for any $i \in [n]$. For the network topology matrix, we use the ER network with $p = 0.6$. For the generalized Stiefel manifold ${\rm St}_B(d,r):={x \in \R^{d\times r}: x^\top B x = I_r}$, the projection is the polar decomposition on ${\rm St}_B(d,r)$ defined in \cite{shustin2023riemannian}. We also use it as the retraction operator for DRDGD and DRGTA. The results are presented in Figure \ref{fig:num-gpca}. DPRGT outperforms all other algorithms in terms of the objective function value and the distance to the ground truth.

\begin{figure}[htp]
	\centering
	\includegraphics[width = 0.45 \textwidth]{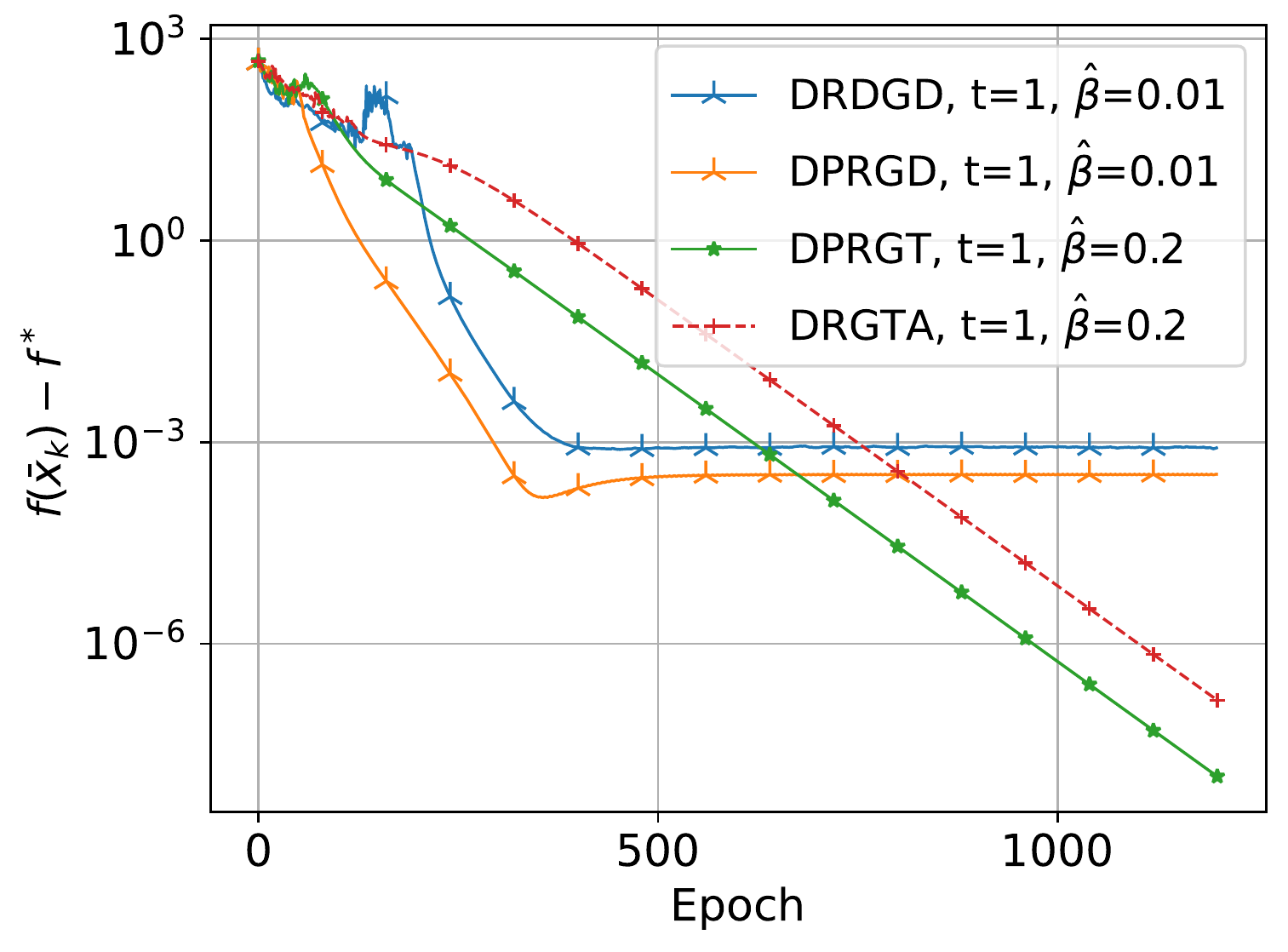}
    \includegraphics[width = 0.45 \textwidth]{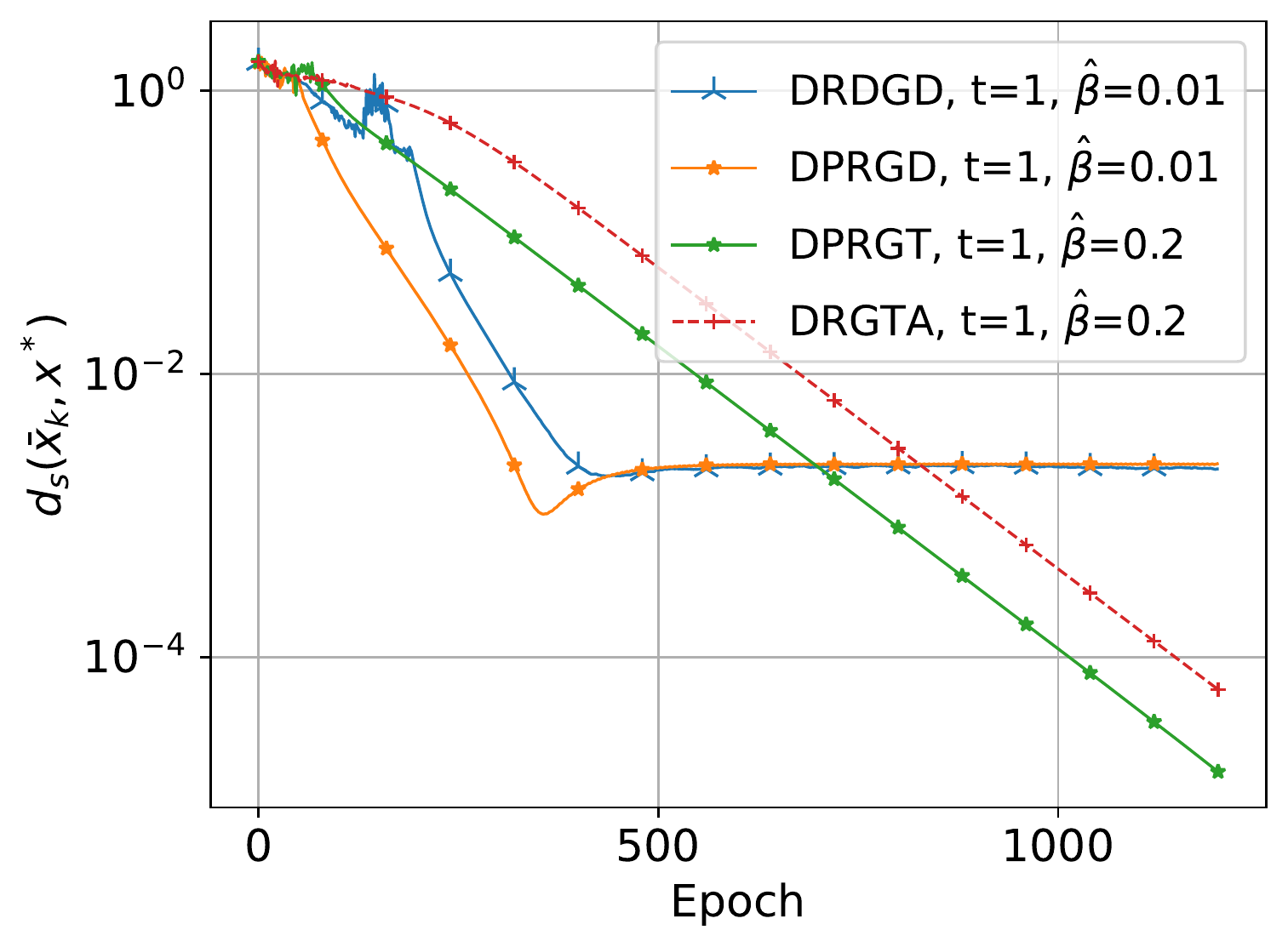}
	\caption{Numerical results for decentralized generalized eigenvalue problem on a synthetic dataset with graph ER $p=0.6$.}
	\label{fig:num-gpca}	
\end{figure}

\subsection{Decentralized low-rank matrix completion}
The goal of the low-rank matrix completion problem (LRMC) is to recover a low-rank matrix $A \in \R^{m \times T}$ from its partial observations. Let $\Omega$ be the set of indices of known entries in $A$, the rank-$r$ LRMC problem can be written as
\be \label{prob:lrmc} \min_{X \in {\rm Gr}(d,r), V\in \R^{r \times T}} \quad \frac{1}{2} \| \Pcal_{\Omega}(XV- A) \|^2, \ee
where ${\rm Gr}(d,r):=\{ {\rm all~} r {\rm ~dimensional~subspaces~in~}\R^d \}$ is the Grassmann manifold and the projection operator $\Pcal_{\Omega}$ is defined in an element-wise manner with $(\Pcal_{\Omega}(A))_{ij} = A_{ij}$ if $(i,j) \in \Omega$ and $0$ otherwise.

For the decentralized setting, let us consider the case when the data matrix $\Pcal_{\Omega}(A)$ is equally divided into $n$ agents by columns, denoted by $A_1, A_2, \ldots, A_n$. As the Grassmann manifold ${\rm Gr}(n,d)$ is a quotient manifold with its total space being ${\rm St}(n,d)$, we can also regard \eqref{prob:lrmc} as an optimization on the Stiefel manifold, which is similar to the case of the principal component analysis in Section \ref{sub:pca}. We refer to \cite{absil2009optimization,hu2022riemannian} for the equivalence of the optimization algorithms between the quotient manifold and its total space.
Then, the decentralized LRMC problem is
\be \label{prob:dlrmc}
\begin{aligned}
\min \quad & \frac{1}{2} \sum_{i=1}^n \| \Pcal_{\Omega_i}(X_i V_i(X) - A_i) \|^2,  \\
\st \quad & X_1 = X_2 = \cdots = X_n, \\
& X_i \in {\rm St}(d,r), \quad \forall i \in [n],
\end{aligned}
\ee
where $\Omega_i$ is the corresponding indices set of $\Omega$ and $V_i(X):= {\rm argmin}_{V} \| \Pcal_{\Omega_i} (XV - A_i)\|$.

For numerical tests, we consider random generated $A$. To be specific, we first generate two random matrices $L \in \R^{m \times r}$ and $R \in \R^{r \times T}$, where each element obeys the standard Gaussian distribution. For the indices set $\Omega$, we generate a random matrix $B$ with each element following from the uniform distribution, then set $\Omega_{ij} = 1$ if $B_{ij} \leq \nu$ and $0$ otherwise. The parameter $\nu$ is set to $r(m+T-r)/(mT)$. In the implementations, we set $T=1000, m=100, r=5$. The fixed step sizes are chosen for all algorithms. The Ring graph is used. We set $\alpha = n \hat{\beta}$ For DPRGT and DRGTA and $\alpha = \frac{\hat{\beta}}{\sqrt{K}}$ with $K$ is the maximal number of iterations for DPRGD and DRDGD. $\hat{\beta}$ is tuned to get the best performance for each algorithm individually.
The results are reported in Figure \ref{fig:num-mc}. We see that DPRGT performs slightly better than DRGTA for both cases. Similarly, DPRGD can achieve lower objective function values compared to DRDGD.
\begin{figure}[htp]
	\centering
	\includegraphics[width = 0.45\textwidth]{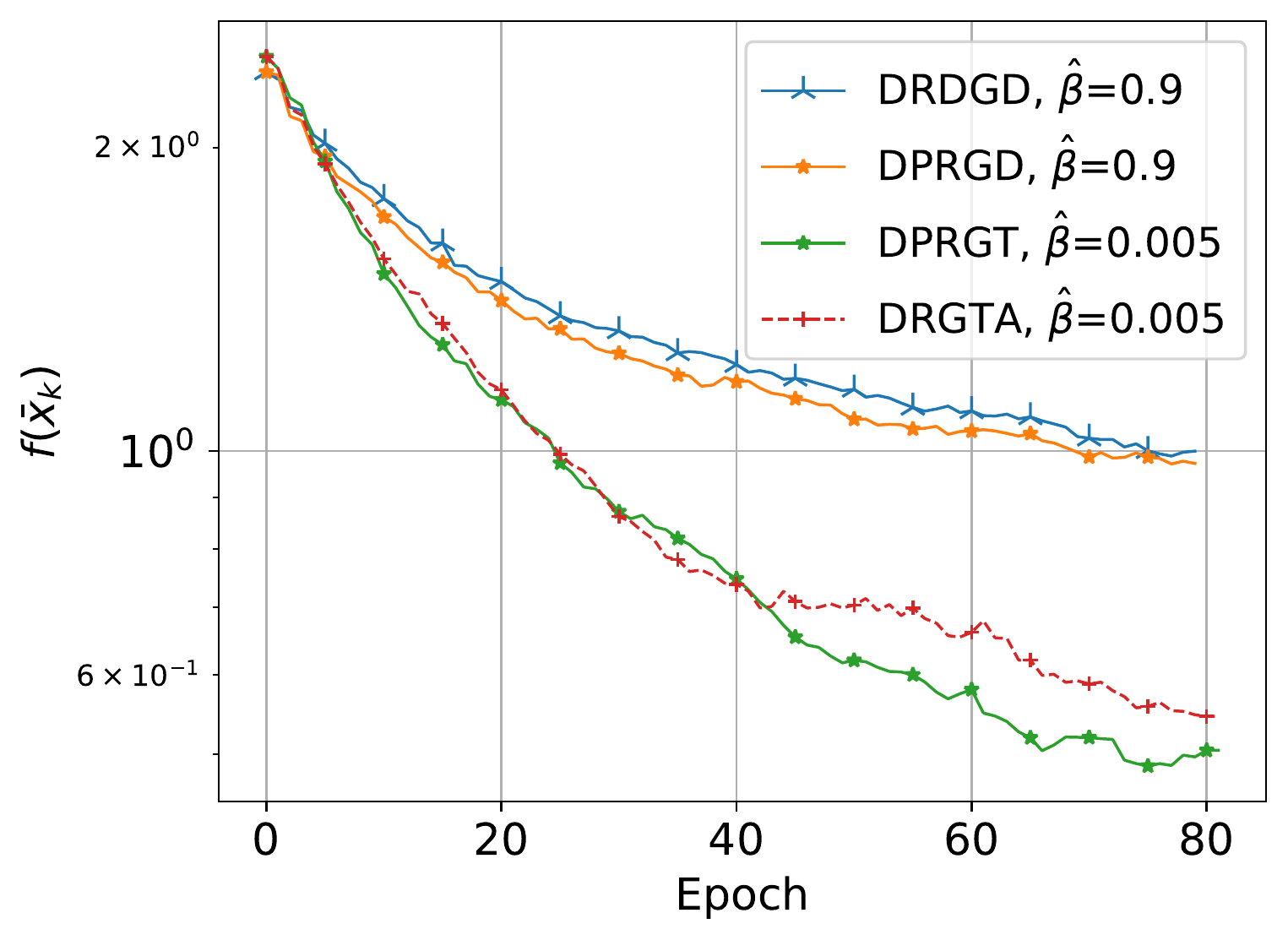}
    \includegraphics[width = 0.45\textwidth]{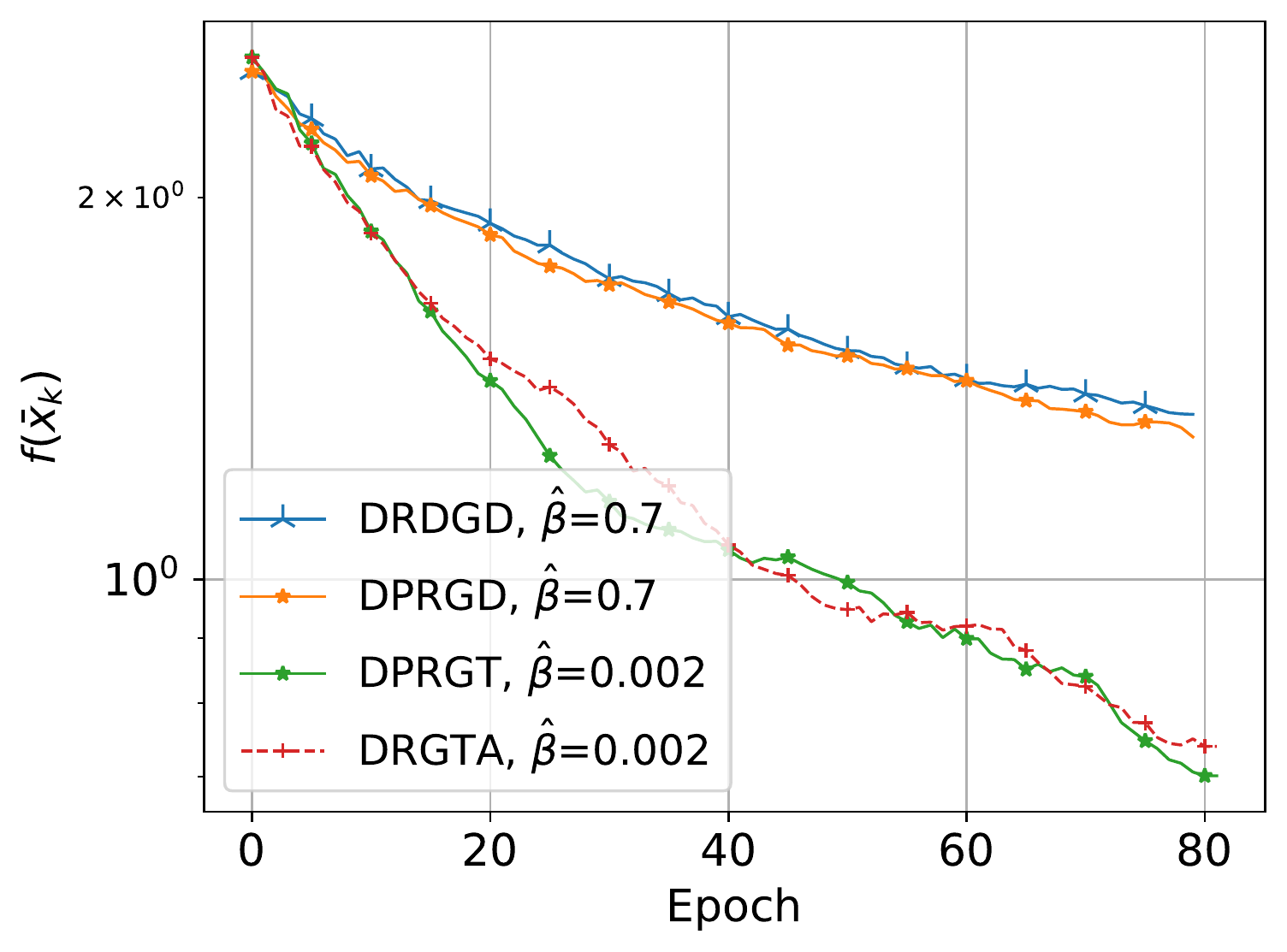}
	\caption{Numerical results for the decentralized LRMC problem with the Ring graph. Left: $n=16$, right: $n=32$.}
	\label{fig:num-mc}	
\end{figure}

\section{Conclusion}
We propose two distributed algorithms for solving decentralized optimization problems over a compact submanifold. To analyze these algorithms, we study the linear convergence of the projected gradient method with unit step size for solving the consensus problem, utilizing the proximal smoothness of the compact submanifold and analyzing several essential Lipschitz-type inequalities. By putting these together, we present the complexity results of the proposed algorithms, which match the best-known results in the Stiefel manifold case \cite{chen2021decentralized}. 
Numerical results demonstrate the effectiveness of our proposed algorithms. It is worth mentioning that the linear convergence of the projected gradient method and the Lipschitz-type inequality established here can be applied to analyze more general manifold optimization and decentralized optimization problems, e.g.,  problem \eqref{prob:original} with weakly convex $f_i$ \cite{wang2023decentralized,peng2022riemannian,wang2023smoothing} and decentralized composite optimization with strongly prox-regular regularizer \cite{hu2023projected}.  

\bibliographystyle{siamplain}
\bibliography{ref}
\end{document}